\documentclass[12pt,authoryear]{elsarticle1}

\usepackage{graphics}
\usepackage{array}
\newcolumntype{P}[1]{>{\centering\arraybackslash}p{#1}}
\newcolumntype{M}[1]{>{\centering\arraybackslash}m{#1}}
\usepackage{amssymb}
\newcommand{\done}{\item[\checkmark]}
\newcommand{\crossed}{\item[$\times$]}
\usepackage{color}
\usepackage[authoryear]{natbib}
\usepackage{subfigure}
\usepackage{graphicx}
\usepackage{amsmath,amsfonts}
\usepackage{amsthm,amssymb}
\renewcommand{\qedsymbol}{\rule{0.7em}{0.7em}}
\usepackage{amsthm,amssymb}
\usepackage{tabularx}
\usepackage{amssymb}
\usepackage{latexsym}
\usepackage{epsfig}

\newcommand{\into}{\hookrightarrow}

\newcommand{\damping}{\mathbf{D}_{0}^{T}\big(g(s); u_{t}\big)}
\newcommand{\dampinglin}{\mathbf{D}_{0}^{T}\big(s;u_{t}\big)}

\newcommand{\tl}{\tilde}
\newcommand{\sgn}{\operatorname{sgn}}

\newcommand{\intX}{\int_{\Omega}}

\newcommand{\ds}{\displaystyle}
\renewcommand{\phi}{\varphi}

\newcommand{\lam}{\lambda}

\newcommand{\cM}{\mathcal{M}}

\newcommand{\cO}{\mathcal{O}}

\newcommand{\bbR}{\mathbb{R}}
\newcommand{\dfn}{:=}

\newcommand{\dX}{\,{\scriptstyle dX }}
\newcommand{\bfD}{\mathbf{D}}

\newcommand{\bfC}{\mathbf{C}}

\newcommand{\del}{\delta}
\renewcommand{\ss}{\scriptstyle}

\newtheorem{corollary}{Corollary}[section]
\newtheorem{definition}[corollary]{Definition}

\newtheorem{lemma}[corollary]{Lemma}

\newtheorem{prp}[corollary]{Proposition}
\newtheorem{remark}[corollary]{Remark}
\newtheorem{thm}[corollary]{Theorem}

\newtheorem{asmp}[corollary]{Assumption}

\def\RR{{\rm I~\hspace{-1.15ex}R} }
\def\CC{\rm \hbox{C\kern-.56em\raise.4ex
		\hbox{$\scriptscriptstyle |$}\kern+0.5 em }}

\def\cM{{\cal M}}

\def\cO{{\cal O}}

\def\be{\begin{equation}}
\def\ee{\end{equation}}

\def\ds{\displaystyle}

\usepackage{amssymb}

\begin{document}

\begin{frontmatter}



\title{\bf Uniform decay rates for a suspension bridge with locally distributed nonlinear damping }

\author[Domingos Cavalcanti]{Andr\'e D. Domingos Cavalcanti}
\address[Domingos Cavalcanti]{Department of Engineering Chemistry,
	State University of Campinas, 13083-970,
	Campinas, SP, Brazil.}
\ead{andre.delano@hotmail.com}

\author[Cavalcanti]{Marcelo M. Cavalcanti\corref{cor1}\fnref{fn1}}
\address[Cavalcanti]{ Department of Mathematics, State University of
	Maring\'a, 87020-900, Maring\'a, PR, Brazil.}
\fntext[Cavalcanti]{Research of Marcelo M. Cavalcanti partially supported by the CNPq Grant
	300631/2003-0}
\ead{mmcavalcanti@uem.br}
\author[Correa,correa]{Wellington J. ~Corr\^ea}

\address[Correa]{ Academic Department of Mathematics, Federal Technological University of  Paran\'a, Campuses Campo Mour\~{a}o,  87301-899, Campo Mour\~{a}o, PR, Brazil.}
\fntext[correa]{Research of Wellington J. Corr\^{e}a partially supported by the CNPq Grant  438807/2018-9}

\ead{wcorrea@utfpr.edu.br}
\author[Hajjej]{Zayd Hajjej}
\address[Hajjej]{Department of Mathematics, Faculty of Sciences of Gabes, University of Gabes, 6029  Gabes, Tunisia. }
\ead{hajjej.zayd@gmail.com}

\author[ci2ma,Sepulv]{Mauricio Sep\'ulveda Cort\'es}
\address[ci2ma]{Centro de Investigaci\'on en Ingenier\'ia Matem\'atica (CI$^2$MA) \& Departamento de Ingenier\'ia matem\'atica (DIM), Universidad de Concepci\'on, Barrio Universitario, Concepci\'on, Chile.}
\fntext[Sepulv]{Research of Mauricio Sep\'ulveda C. was supported FONDECYT grant no. 1180868, and by CONICYT-Chile through the project AFB170001 of the PIA Program: Concurso Apoyo a Centros Cient\'ificos y
	Tecnol\'ogicos de Excelencia con Financiamiento Basal.}
\ead{mauricio@ing-mat.udec.cl}

\author[ci2ma,RVA]{Rodrigo V\'ejar Asem}
\fntext[RVA]{Rodrigo V\'ejar Asem, PhD student at University of Concepci\'on, acknowledges support by CONICYT-PCHA/Doctorado Nacional/2015-21150799.}
\ead{rodrigovejar@ing-mat.udec.cl}

\cortext[cor1]{Corresponding author}

\begin{abstract}
		We study a nonlocal evolution equation modeling the deformation of a bridge, either a footbridge or a suspension bridge. Contrarily to the previous literature we prove the asymptotic stability of the considered model with a minimum amount of damping which represents less cost of material. The result is also numerically proved.		\end{abstract}


\end{frontmatter}
\tableofcontents






\section{Introduction}

\subsection{Statement of the problem and literature overview}

In the present paper, inspired by the works of Al-Gwaiz,  Benci, Ferrero, Gazzola et. al \cite{Gazzola1'}, \cite{Gazzola1},\cite{Gazzola2} (and references therein \cite{Berger}, \cite{Burgreen}, \cite{Knightly}, \cite{Mansfield}, \cite{Ventsel}, \cite{Villaggio}, \cite{Woinowsky-Krieger}) we consider a thin and narrow rectangular plate where the two short edges are hinged whereas the two long edges are free. This plate aims to represent the deck of a bridge, either a footbridge or a suspension bridge. In absence of forces, the plate lies flat horizontally and is represented by the planar domain $\Omega=(0,\pi)\times (-l, l)$ where $l<<\pi,$ with boundary $\Gamma$. Then, the nonlocal evolution equation modeling the deformation of the plate reads as follows:
\begin{equation}\small \label{main problem}
\hspace*{-.2cm}
\begin{cases}
u_{tt}(x,y,t)+\Delta^2 u(x,y,t) + \phi(u) u_{xx} + a(x,y)g(u_t(x,y,t))=h,\quad  \mbox{ in } \;\Omega \times (0,+\infty), \\\\

u(0,y,t)=u_{xx}(0,y,t)=u(\pi,y,t)=u_{xx}(\pi,y,t)=0, \quad \quad  \!\!\!\!(y,t)\in(-l,l)\times (0,+\infty), \\\\

u_{yy}(x,\pm l,t)+\sigma u_{xx}(x,\pm l, t)=0, \quad \qquad \qquad \qquad \qquad \quad (x,t)\in (0,\pi)\times (0, +\infty), \\\\

u_{yyy}(x,\pm l,t)+(2-\sigma)u_{xxy}(x,\pm l, t)=0,
\quad \qquad \qquad \qquad \!\! (x,t)\in(0,\pi)\times(0,+\infty),\\\\

u(x,y,0)=u_0(x,y),\; u_t(x,y,0)=u_1(x,y), \quad \qquad \qquad \qquad \qquad \quad \qquad \qquad \mbox{ in } \;\Omega,

\end{cases}
\end{equation}
where the nonlinear term $\phi$, which carries a nonlocal effect into the model, is defined by
\begin{eqnarray*}
	\phi(u)=-P + S \int_\Omega u_x^2\;dx,
\end{eqnarray*}
where the constant $\sigma$ is the Poisson ratio: for metals its value lies around $0.3$ while for concrete it is between $0.1$ and $0.2$. For this reason we shall assume that $0<\sigma <\frac{1}{2}$,  $a=a(x,y)\in L^{\infty}(\Omega)$ is assumed to be a nonnegative essentially bounded function such that $$a\geq a_0>0\;\;\hbox{ a.e. in } \omega,$$ for some non empty open subset $\omega$ around the boundary $\Gamma$ of $\Omega$ and some positive constant $a_0>0$ and the function $g$ verifies such conditions that be announced in the third section. \\
Here $S>0$ depends on the elasticity of the material composing the deck, the term $S \displaystyle{\int_\Omega} u_x^2\;dx$ measures the geometric nonlinearity of the plate due to its stretching, and $P>0$ is the prestressing constant: one has $P>0$ if the plate is compressed and $P<0$ if the plate is stretched. The function $h$ represents the vertical load over the deck and may depend on time.\\  Early results concerning suspension bridges go back \cite{Glover} and for rigid suspension bridges it is worth mentioning \cite{Lazer}. Mckenna and Walter \cite{McKenna}, \cite{McKenna2} investigated the nonlinear oscillations of suspension bridges and the existence of traveling wave solutions have been established. To achieve this, they considered the suspension bridge as a vibrating beam as in the present paper.  {\color{black} Recently, there has been a lot of work on the bridge configuration [\cite{Gazzola1'}, \cite{Gazzola1},\cite{Gazzola2}] with this type of Berger's nonlinearity. The key feature in the present paper is the localized damping and the rectangular geometry. Also, the nonlinearity acts as a beam (only in the span direction), but the model does allow for dynamics in the torsional sense.}

We start talking about the resonance phenomenon in bridges and buildings and the importance of dampers to prevent dangers and faults in constructions's structures . Resonance is the reinforcement or prolongation of sound by reflection from surface or by the synchronous vibration of a neighboring object. In simpler terms, the conditions which the frequency of a wave equals the resonant frequency of the waves medium. Mechanical resonance occurs when there is transfer of energy from one object to another with the same natural or resonant frequency. Strong vibrations can cause lots of damage to structures and can be used to break materials apart. The main reason for the Tacoma Narrows Bridge collapse (see figure \ref{tacoma}) was the sudden transition from longitudinal to torsional oscillations caused by resonance phenomenon. Several other bridges collapsed for the same reason (see \cite{Amman}, \cite{Scott}). In \cite{Gazzola2} the authors analyze in detail how a solution of (\ref{main problem}) initially oscillating in an almost purely longitudinal fashion can suddenly start oscillating in a torsional fashion, even without the interaction of external forces, that is, when $h=0$. For this reason we shall consider $h=0$ in the present manuscript. {\color{black} As a matter of fact, although the collapse Tacoma bridge's is a matter of debate until now, the resonance phenomenon causes irreparable damages in constructions and the presence of dampers play an essential role in stabilizing bridges and other constructions whatever the reason which the bridge's collapse.}

\begin{figure}[h!]\begin{center}
		\includegraphics[scale=0.5]{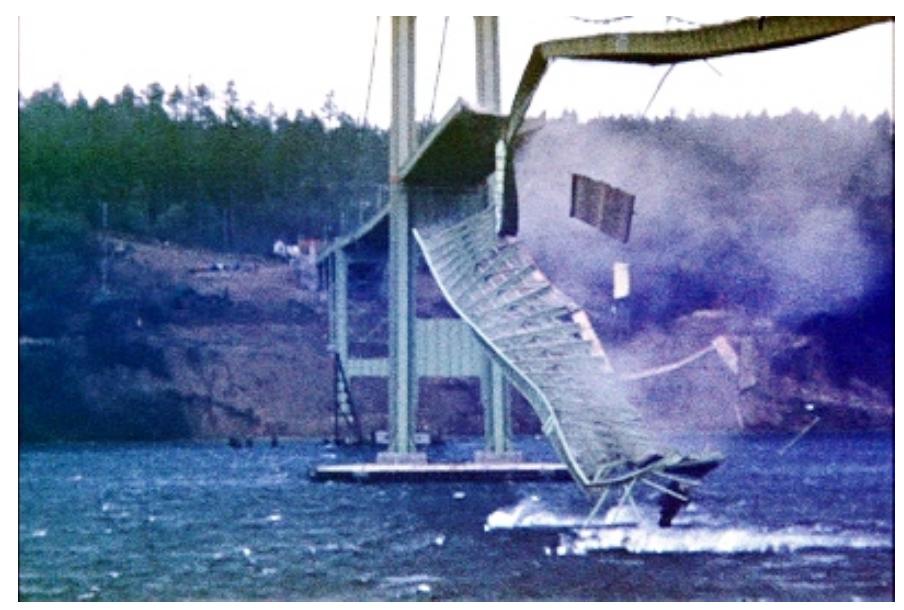}
		\caption{Tacoma Narrows Bridge collapse}
		\label{tacoma}
	\end{center}
\end{figure}

To limit unwanted vibrations and preventing structures from resonating with frequencies during earthquakes, features and modifications such as {\em dampers} are designed to help us in that way (see figure \ref{damper}). They help save buildings or bridges from damage costs, and lives of people. Understanding how vibrations work can help us prevent dangers and faults in structures and natural disasters. Over the year engineers have discovered ways and have made design modifications to bridges and buildings to help limit undesired vibrations. This helps structures from shaking too much and causing them to be unsafe or from collapsing due to strong natural forces. One way to limit vibrations include the use of dampers. Damping is the reduction in the amplitude of a wave as a result of energy absorption destructive interference. Seismic dampers are a type of dampers and are mechanical devices to dissipate kinetic energy of seismic waves penetrating a building or a bridge structure. Tune dampers are another kind of dampers, also known as a harmonic absorber, is a device mounted in structures to reduce the amplitude of mechanical vibrations. Their application can prevent discomfort, damage, or outright structural failure. They are frequently used in power transmission, bridges and buildings. {\color{black} From the mathematical and physical point of view the damping term above mentioned is represented by the term $a(x,y) g(u_t)$ where the function $a(x,y)$, assumed to be non negative, is effectively responsible by the location where the nonlinear damping $g(u_t)$ acts on the structure, that is, $a(x,y) >0$ in a neighbourhood $\omega$ around the boundary $\Gamma$ of $\Omega$ where the damping term is effective and $a(x,y)=0$ in $\Omega\backslash \omega$ so that no mechanism of damping is acting in the structure.}

\begin{figure}[h!]\begin{center}
		\includegraphics[scale=0.3]{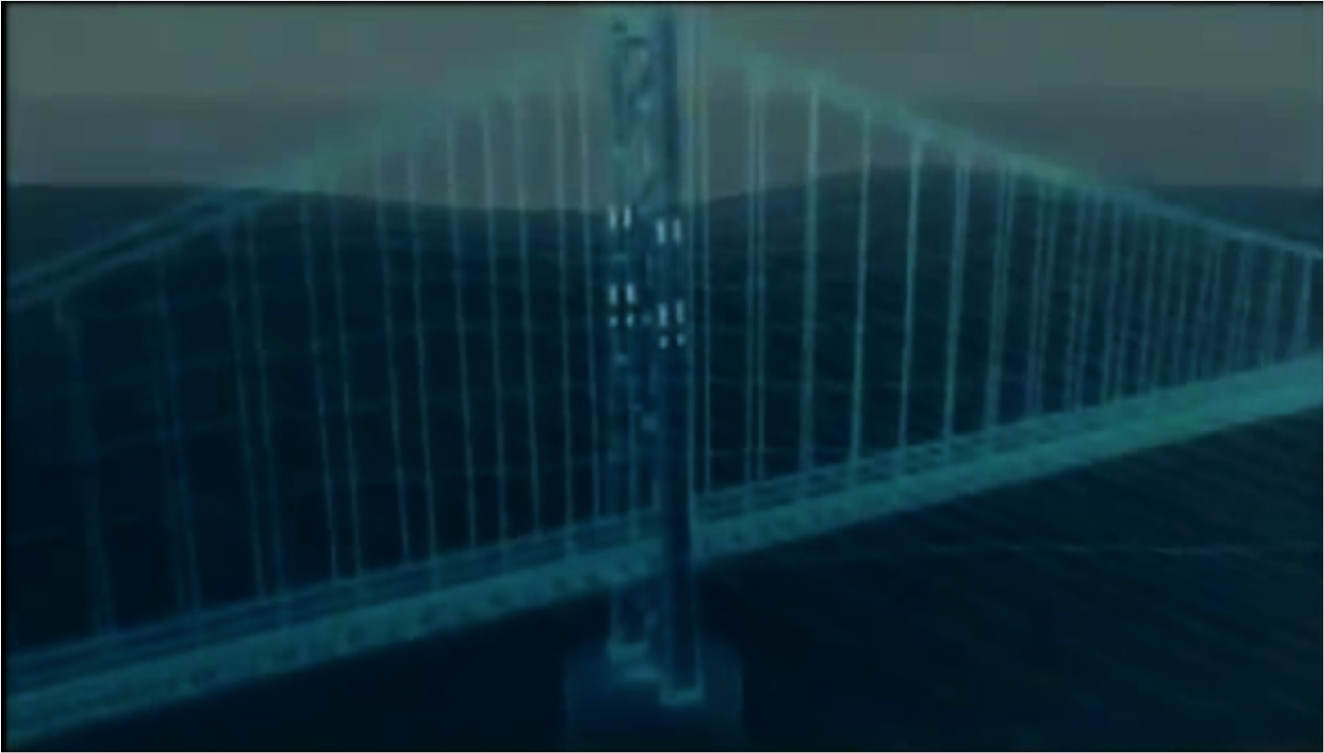}\quad
		\includegraphics[scale=0.3]{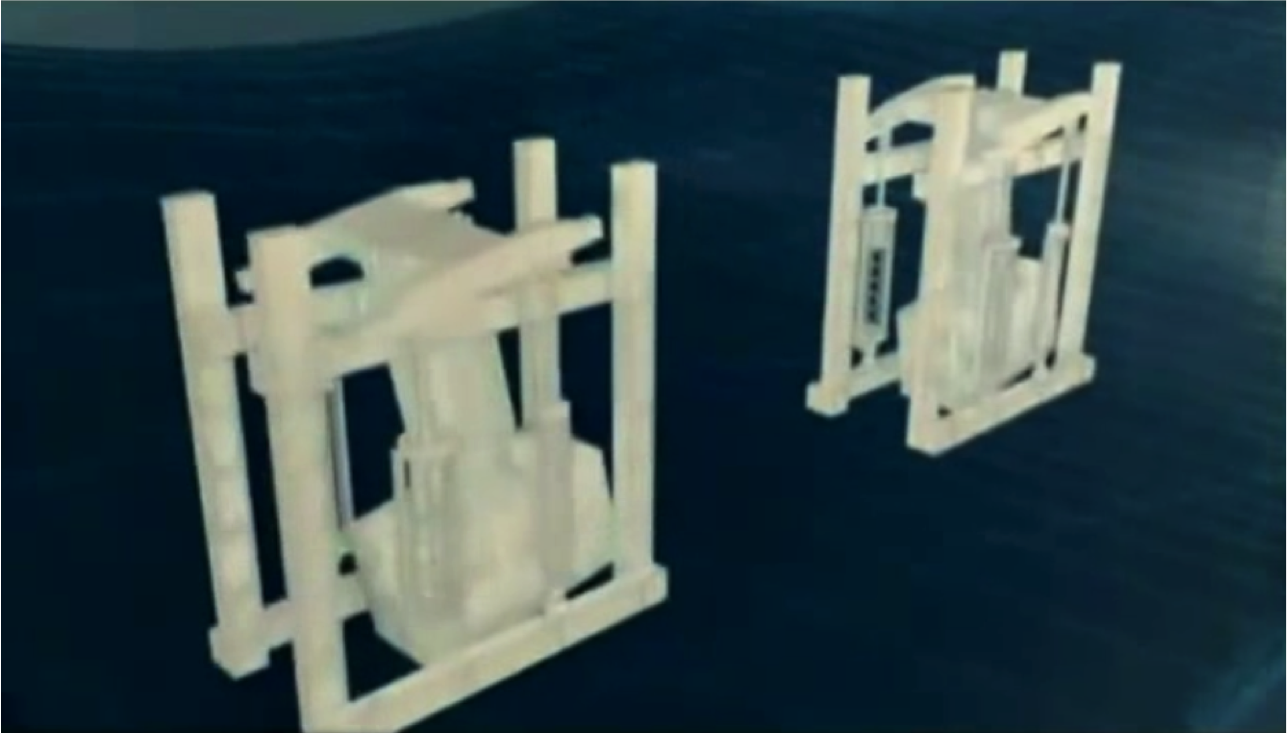}
		\caption{Dampers prevent sudden transition from longitudinal to torsional oscillations caused by resonance phenomenon}
		\label{damper}
	\end{center}
\end{figure}

\subsection{Contribution of the present article}

The main goal of the present article is to establish uniform decay rates estimates to problem (\ref{main problem}) with a minimum amount of damping which represents less cost of material.  {\color{black} This minimum refers a small `collar' $\omega$ around the whole boundary $\Gamma$ of $\Omega$.  In addition, the nonlinear feedback $a(x,y) g(u_t)$ can be superlinear, sublinear or linearly bounded at infinity according to the terminology given in (\ref{i:order})}. \cite{Bochicchio} considered a similar model as in (\ref{main problem}) where a full damping is in place and they established a well-posedness result as well as the existence of a global attractor.  \cite{Messaoudi} reformulate (\ref{main problem}), with a different kind of nonlinearity, into a semigroup setting and then make use of the semigroup theory to establish the well-posedness. They also use the multiplier method to prove an exponential stability result to problem (\ref{main problem}) when also a full damping is in place, namely, when $\delta(x,y)=\delta >0$. More recently,   \cite{Gazzola2} study the same nonlocal evolution equation (with $\delta(x,y)=\delta >0$) and prove existence, uniqueness and asymptotic behavior for the solutions for all initial data in suitable functional spaces. Further, the authors prove results on the stability/instability of simple models motivated by a phenomenon which is visible in actual bridges and they complement their study with some numerical experiments.

As far as we are concerned there are few papers which deal with the asymptotic dynamics to problem (\ref{main problem}) and the present paper seems to be the pioneer in investigating the asymptotic stability to problem (\ref{main problem}) with a nonlinear damping locally distributed just around a neighbourhood $\omega$ of the boundary $\Gamma$ of $\Omega$. First we prove the observability inequality associated to the linear model without damping.  For this purpose we make use of the multiplier method see, for instance, \cite{Komornik}, \cite{Lions}, usually adopted for plates and beam equations, now adjusted to the present model, which brings new difficulties to be overcome {\color{black} because of the `hard terms' which come from the boundary conditions and mainly due to the lack of an unique continuation principle for domains with non smooth boundary}. Second, exploiting the observability inequality above mentioned we deduce uniform decay rate estimates of the energy correspondent nonlinear model. For the nonlinear model we borrow ideas firstly introduced in  \cite{Tucsnak} mainly how to be succeed in using a unique continuation property due to  \cite{Kim}. However, due to the shape of our domain (non smooth boundary) new difficulties appear which were overcome by using of geometric tools. {\color{black} Indeed, when a full damping is in place, as have been considered previously in the literature so far, the multipliers used to derive decay rate estimates are easier to be controlled and no unique continuation principle is required. However, when one has a nonlinear damping locally distributed in a `collar' of the boundary, we need to consider non radial multipliers as previously considered by \cite{Tucsnak} for the nonlinear beam equation subject to boundary conditions $u=\partial_\nu u=0 $ on $\Gamma \times \mathbb{R}_+$ where $\Gamma$ is smooth. But the boundary conditions concerned to the present article are very complicated to the handled from the technical point of view. Furthermore, definitively one the major difficulty found in the present article was to extend the unique continuation principle proved in \cite{Kim} for domains with smooth boundary to the present case where the boundary contains corners. The strategy used to overcome this difficulty is to consider a sequence of sub domains $\Omega_{\epsilon_n}$ with smooth boundary where the unique continuation principle holds and since $\Omega_{\epsilon_n}$ converges to $\Omega$ uniformly when $\epsilon_n$ converges to zero the unique continuation principle remains valid for the rectangle $\Omega$ ( see figure \ref{fig 5}). We believe that the strategy used in the present article will be useful for other models in which the boundary is not necessarily smooth. Summarizing, the main contribution of the present work is represented by the technical challenges induced by the configuration ({\em free boundary condition}) and the {\em rectangular geometry} of the problem.} The second section of the present paper is devoted to the linear case while the third one we analyze the nonlinear model. {\color{black} The fourth section proves the energy decay estimates. The fifth section replicates the second main contribution of this paper using a finite difference scheme, whose main advantage relies on a practical and simple way to implement it using a matrix-based programming language like \texttt{MATLAB}. The last section is an appendix containing the geometric tools needed for the proof.}


\section{The linear Model}

\subsection{Notation and Preliminary Results}

We  consider the following system

\begin{equation}\small\label{1}\hspace*{-.1cm}
\begin{cases}
u_{tt}(x,y,t)+\Delta^2 u(x,y,t)= 0, &\mbox{ in }\Omega \times (0,+\infty), \\\\

u(0,y,t)=u_{xx}(0,y,t)=u(\pi,y,t)=u_{xx}(\pi,y,t)=0, &(y,t)\in(-l,l)\times (0,+\infty), \\\\

u_{yy}(x,\pm l,t)+\sigma u_{xx}(x,\pm l, t)=0, &(x,t)\in (0,\pi)\times (0, +\infty), \\\\

u_{yyy}(x,\pm l,t)+(2-\sigma)u_{xxy}(x,\pm l, t)=0, &(x,t)\in(0,\pi)\times(0,+\infty),\\\\

u(x,y,0)=u_0(x,y),\; u_t(x,y,0)=u_1(x,y), &\mbox{ in } \;\Omega.
\end{cases}
\end{equation}

We introduce the space
$$ H^2_*(\Omega)=\{w\in H^2(\Omega): w=0 \;{\text on}\; \{0,\pi\}\times (-l, l)\},$$
together with the inner product
$$(u,v)_{H^2_*}=\int_{\Omega} F(u,v) dx dy,$$
where
\begin{equation}\label{01}
F(u,v)=u_{xx}v_{xx}+u_{yy}v_{yy}+\sigma (u_{xx}v_{yy}+u_{yy}v_{xx})+2(1-\sigma)u_{xy}v_{xy}.
\end{equation}
It is well known that $(H^2_*(\Omega), (\cdot, \cdot)_{H^2_*})$ is a Hilbert space, and the norm $\Vert . \Vert^2_{H^2_*} $ is equivalent
to the usual $H^2$ norm (see \cite{Gazzola1}).\\

Introducing the following notation
\begin{eqnarray}\label{001}
\left\{
\begin{array}{lcr}
u_{xx}(0,y)=u_{xx}(\pi,y)=0, \\

u_{yy}(x,\pm l)+\sigma u_{xx}(x,\pm l)=0, \\

u_{yyy}(x,\pm l)+(2-\sigma)u_{xxy}(x,\pm l)=0.
\end{array}
\right.
\end{eqnarray}

We have
\begin{lemma}[\cite{Messaoudi}]
	\begin{equation}\label{02}
	(\Delta^2 u, v)_{L^2(\Omega)}= (u,v)_{H^2_*(\Omega)}=\int_\Omega F(u, u)\;dx\;dy,
	\end{equation}
	$\forall\;u\in H^4(\Omega)\cap H^2_*(\Omega)\;\;{\text satisfying}\;\;(\ref{001}), \;{\text and}\; v\in H^2_*(\Omega).$
\end{lemma}

Problem (\ref{1}) can be written
\begin{eqnarray*}
	\left\{
	\begin{aligned}
		& U_t + A U =0,\\
		& U(0)=U_0,
	\end{aligned}
	\right.
\end{eqnarray*}
where
\begin{eqnarray*}
	U=\left(
	\begin{aligned}
		&u\\
		&v
	\end{aligned}
	\right);\quad
	A U:=\left(
	\begin{aligned}
		& -v\\
		& \Delta^2u
	\end{aligned}
	\right);\quad
	U_0=\left(
	\begin{aligned}
		&u_0\\
		&v_0
	\end{aligned}
	\right).
\end{eqnarray*}

We define the Hilbert space $\mathcal{H}:= H_*^2(\Omega)\times L^2(\Omega)$ endowed with the inner product
\begin{eqnarray*}
	(U,V)_{\mathcal{H}}=(u,\tilde u)_{H^2_*(\Omega)} + (v, \tilde v)_{L^2(\Omega)},
\end{eqnarray*}
where $U=(u,v)^T$; $V=(\tilde u, \tilde v)^T\in \mathcal{H}$. The domain of the operator $A$ is defined by
\begin{eqnarray*}
	D(A):= \{(u,v)\in \mathcal{H}: u\in H^4(\Omega) \hbox{ satisfying } \eqref{001}, v\in H^2_*(\Omega) \}.
\end{eqnarray*}

The wellposedness of problem (\ref{1}) can be  studied  as in \cite{Messaoudi}. Indeed, let $U_0 \in \mathcal{H}$ given, Then as in  \cite{Messaoudi} (Theorem 3.1) problem (\ref{1}) possesses a unique solution $U\in C([0,\infty);\mathcal{H})$. In addition, if $U_0  \in D(A)$, then problem (\ref{1}) has a unique regular solution $U\in C([0+\infty);D(A)) \cap C^1([0,\infty);\mathcal{H})$.

\subsection{Observability Inequality}

We define the energy of solutions of system (\ref{1}) by:
$$E(t)=\frac{1}{2} \Vert u_t (t)\Vert^2_{L^2}+\frac{1}{2} \Vert u (t)\Vert^2_{H^2_*}, ~t\geq 0.$$

The following identity holds
$$\frac{d E(t)}{dt}= 0,~\hbox{ for all }t\geq 0,$$
from which we deduce the identity of the energy
\begin{eqnarray}\label{Linear Idend Ener}
E(t)=E(0), \hbox{ for all }t\geq 0.
\end{eqnarray}

The aim of this section is to give sufficient conditions on  $\omega$ ensuring the observability inequality holds
for every solution of (\ref{1}), when $\omega$ is a neighbourhood of the boundary $\Gamma$.\\
Namely, it suffices to prove the existence of a positive constants $C, T_0$ such that
\begin{eqnarray}\label{3}
E(0)\leq C\int_0^T\int_{\Omega} \chi_\omega \vert u_t(x,y,t)\vert^2 dx\;dy\;dt, ~\forall\; T\geq T_0,
\end{eqnarray}
where $\chi_\omega$ represents the characteristic function of $\omega$. We have the following theorem:
\begin{thm}\label{Teo1}
	For any $L>0$ there exist positive constants $C$ and $T_0$, such that, if $E(0)\leq L$ then (\ref{3}) holds true.
\end{thm}
\begin{remark}\label{remark 1}
	It is worth mentioning that we could carry in the linear problem (\ref{1}), as well as in Theorem \ref{Teo1}, the linear part $- P u_{xx}$ of the nonlinear term $\phi(u)u_{xx}:= - P u_{xx} + S \displaystyle{\int_\Omega} u_x^2\,dx \,u_{xx}$ since it does not affect the proof given in the sequel. For simplicity we decided to remove it.
\end{remark}
\noindent{\bf Proof:}
	The proof of Theorem \ref{Teo1} consists of three steps.\\
	\underline{Step 1}\;\; We shall work with regular solutions and by standard density arguments the inequality (\ref{3})
	remains valid for weak solutions as well. Let us multiply the first equation in (\ref{1}) by  $q(x,y)\,\cdot\,\nabla u(x,y,t)$ where $q\in (W^{2,\infty})^2$ (we denote by $\cdot$ the scaler product in $\mathbb{R}^2$).\\
	Following the integrations by parts of Lemma 3.3 (see \cite{Lions} p. 244, see also \cite{Tucsnak}) adapted to the present case, we obtain:
{\small	\begin{eqnarray}
&&		\hspace*{-.2cm}	\left[ \int_\Omega (u_t q\,\cdot\,\nabla u)dx\;dy \right]^T_0 +\frac{1}{2}\int_Q div(q)  \vert u_t\vert^2\; dx\;dy\;dt +\int_Q (\Delta q\,\cdot\,\nabla u)\Delta u\;dx\;dy\;dt\nonumber\\
	&&		+2\sum_{j,k} \int_Q \frac{\partial {q_k}}{\partial x_k} \Delta u \frac{\partial^2 u}{\partial x_k \partial x_j} \; dx\;dy\; dt -\frac{1}{2}\int_Q div(q)  \vert\Delta u\vert^2\; dx\;dy\;dt\nonumber\\
	&&		=\frac{1}{2}\int_{\Sigma} (q\,\cdot\,\nu) \vert u_t\vert^2 d\Gamma\; dt -\frac{1}{2}\int_{\Sigma} (q\,\cdot\,\nu) \vert\Delta u\vert^2\;d\Gamma\;dt-\int_{\Sigma} \partial_\nu \Delta u (q\,\cdot\,\nabla u)  d\Gamma\; dt \nonumber\\
	&&	+\sum_{k=1}^{2} \int_{\Sigma} \partial_\nu q_k \Delta u \frac{\partial u}{\partial x_k}  d\Gamma\; dt + \int_{\Sigma}\Delta u (q\,\cdot\,\partial_\nu \nabla u)  d\Gamma\; dt\,.\label{4}
	\end{eqnarray}}
	where $Q=\Omega\times (0,T)$ and $\Sigma=\Gamma\times (0,T)$.\\
	
	Applying identity (\ref{4}) with $q(x,y)=m(x,y)=X-X_0$ for some $X_0\in\mathbb{R}^2$, where $X=(x,y)$, we obtain
{\small\hspace*{-.5cm}	\begin{eqnarray}
	&&	\!\!\!	\left[ \int_\Omega (u_t m\,\cdot\,\nabla u)dx\;dy \right]^T_0 +\int_Q  \vert u_t\vert^2\; dx\;dy\;dt + \int_Q  \vert \Delta u\vert^2 \; dx\;dy\; dt\nonumber\\ 
	&&		=\frac{1}{2}\int_{\Sigma} (m\,\cdot\,\nu) \vert u_t\vert^2\; d\Gamma\; dt -\frac{1}{2}\int_{\Sigma} (m\,\cdot\,\nu) \vert\Delta u\vert^2\;d\Gamma\;dt-\int_{\Sigma} \partial_\nu \Delta u (m\,\cdot\,\nabla u)\  d\Gamma\ dt \nonumber\\
	&& 	+\int_{\Sigma} \partial_\nu u\; \Delta u\ d\Gamma\; dt + \int_{\Sigma}\Delta u (m\,\cdot\,\partial_\nu \nabla u)\  d\Gamma\  dt,\label{5}
	\end{eqnarray}}
	where  we used that\\	
	\begin{eqnarray*}
		\displaystyle{\sum_{k=1}^{2}} \int_{\Sigma} \partial_\nu m_k \Delta u \frac{\partial u}{\partial x_k}\;  d\Gamma\; dt= \displaystyle{\sum_{k=1}^{2}} \int_{\Sigma} \Delta u \; \nu_k \frac{\partial u}{\partial x_k}\;  d\Gamma\; dt=\int_{\Sigma}  \partial_\nu u\; \Delta u\;  d\Gamma\; dt\,.
	\end{eqnarray*}	
	
	Let us now multiply the first equation of (\ref{1}) by $u$  and integrate over $Q$, we get
	\begin{eqnarray}\label{6}
	\int_\Omega [u u_t]dx\;dy|^T_0 + \int_0^T \Vert u\Vert^2_{H^2_*}-\int_Q \vert u_t\vert^2 dx\;dy\;dt =0.
	\end{eqnarray}
	
	Multiplying (\ref{6}) by $0<\alpha<1$ and taking the sum with  (\ref{5}), we get
	{\small \begin{eqnarray}
	&&		\left[\alpha \int_\Omega uu_t \,dx\;dy +\int_\Omega (u_t m\,\cdot\,\nabla u)dx\;dy\right]^T_0 +(1-\alpha)\int_Q   \vert u_t\vert^2\; dx\;dy\;dt  \nonumber\\
	&&	+\;\alpha\int_Q \Vert u\Vert^2_{H^2_*}\;dx\;dy\;dt+\int_Q \vert\Delta u\vert^2\;dx\;dy\;dt\nonumber\\ 
	&&		=\frac{1}{2}\int_{\Sigma} (m\,\cdot\,\nu) \vert u_t\vert^2\; d\Gamma\; dt -\frac{1}{2}\int_{\Sigma} (m\,\cdot\,\nu) \vert\Delta u\vert^2\;d\Gamma\; dt-\int_{\Sigma} \partial_\nu \Delta u (m\,\cdot\,\nabla u)\;  d\Gamma\; dt\nonumber\\
	&&	+\int_{\Sigma} \partial_\nu u\; \Delta u\;   d\Gamma\; dt + \int_{\Sigma}\Delta u (m\,\cdot\,\partial_\nu \nabla u)\;  d\Gamma\; dt.\label{7}
	\end{eqnarray}}

	Let $\varepsilon>0$ small enough, taking $\Omega_\varepsilon= (\varepsilon, \pi -\varepsilon)\times (-l+\varepsilon, l-\varepsilon)$, and let us  define the cutoff function (see figure \ref{fig 3} ):
	\begin{eqnarray}\label{8}
	\left\{
	\begin{array}{lcr}
	\psi=1\; &{\text in}\; \Omega\setminus \Omega_\varepsilon,&\\
	0\leq \psi\leq 1\;\; &{\text in}\;\; \Omega_\varepsilon\setminus \Omega_{2\varepsilon},&\\
	\psi=0 \;\; &{\text in}\;\; \Omega_{2\varepsilon},&
	\end{array}
	\right.
	\end{eqnarray}
	such that  $(\Omega\setminus \Omega_{2\varepsilon})\subset\subset \omega$, so that we have damping in $\Omega\setminus \Omega_{2\varepsilon}$.
	
	\medskip
	
	\begin{figure}[h!]
		\centering

\includegraphics[width=1.1\linewidth]{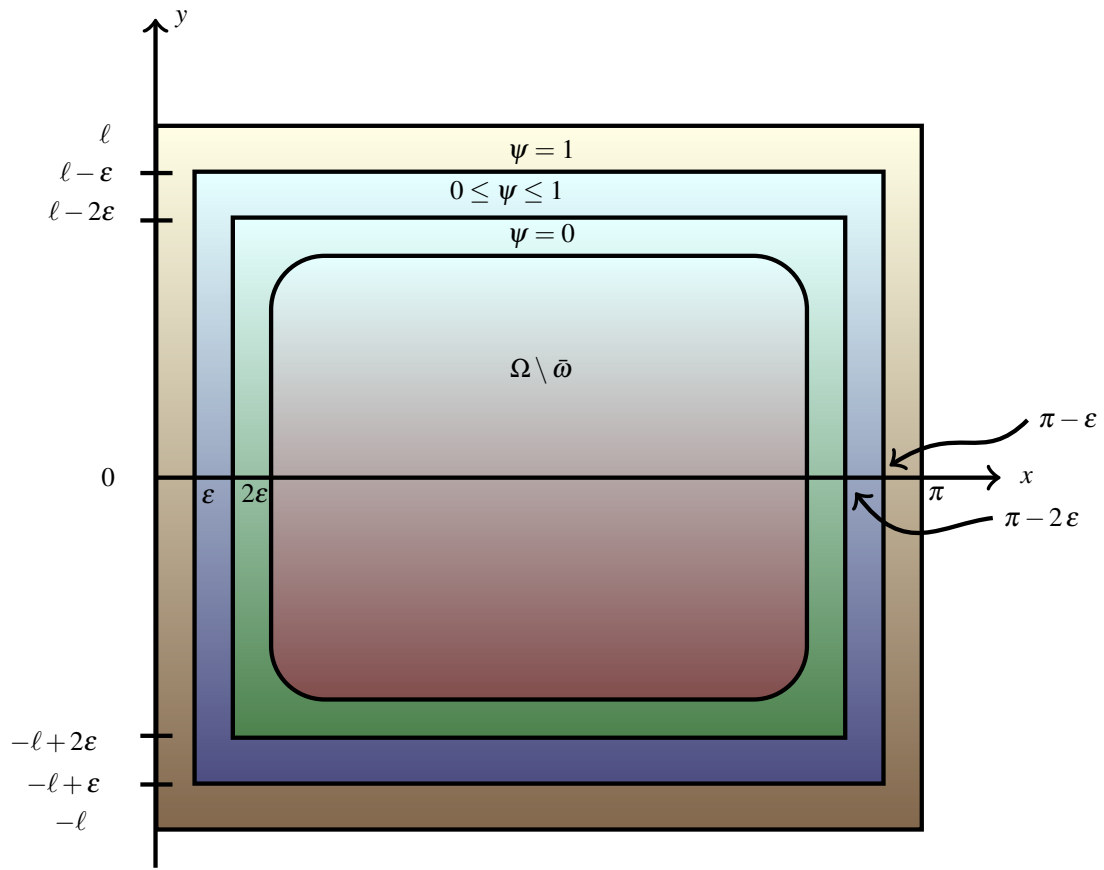}

		\caption{Function $\psi$}
		\label{fig 3}
	\end{figure}
	

	Now taking $q=\psi m$ in (\ref{4}), we get
	{\small \begin{eqnarray}
	&&		\left[ \int_\Omega u_t \; \psi m\,\cdot\,\nabla udx\;dy \right]^T_0 +\frac{1}{2}\int_Q div(\psi m)  \vert u_t\vert^2\; dx\;dy\;dt\nonumber\\
	&&  +\int_Q (\Delta (\psi m)\,\cdot\,\nabla u) \; \Delta u \;dx\;dy\;dt		+2\sum_{j,k} \int_Q \frac{\partial (\psi m)_k}{\partial x_j} \Delta u \frac{\partial^2 u}{\partial x_k \partial x_j}\;dx\;dy\; dt\nonumber\\
	&&-\frac{1}{2} \int_Q  div(\psi m) \vert\Delta u\vert^2\; dx\;dy\;dt\nonumber\\
	&&		=\frac{1}{2}\int_{\Sigma} (\psi m\,\cdot\,\nu) \vert u_t\vert^2\; d\Gamma\; dt -\frac{1}{2}\int_{\Sigma} (\psi m\,\cdot\,\nu) \vert\Delta u\vert^2\;d\Gamma\;dt\nonumber\\
	&& -\int_{\Sigma} \partial_\nu \Delta u (\psi m\,\cdot\,\nabla u)\;  d\Gamma\; dt	+\sum_{k=1}^{2}\int_{\Sigma} \partial_\nu (\psi m)_k\;\Delta u\;\frac{\partial u}{\partial x_k}\;   d\Gamma\; dt \nonumber\\
	&& + \int_{\Sigma}\Delta u (\psi m\,\cdot\,\partial_\nu \nabla u)\;  d\Gamma\; dt.\label{9}
	\end{eqnarray}}
	
	Since $\psi=1$ on $\Sigma$, one has:
	{\small \begin{eqnarray}
	&&		\frac{1}{2}\int_{\Sigma} (\psi m\,\cdot\,\nu) \vert u_t\vert^2\; d\Gamma\; dt -\frac{1}{2}\int_{\Sigma} (\psi m\,\cdot\,\nu) \vert\Delta u\vert^2-\int_{\Sigma} \partial_\nu \Delta u (\psi m\,\cdot\,\nabla u)\;  d\Gamma\; dt\nonumber\\
	&&		+\sum_{k=1}^{2}\int_{\Sigma} \partial_\nu (\psi m)_k\;\Delta u\;\frac{\partial u}{\partial x_k}\;   d\Gamma\; dt + \int_{\Sigma}\Delta u (\psi m\,\cdot\,\partial_\nu \nabla u)\;  d\Gamma\; dt,\nonumber\\
	&&		=\frac{1}{2}\int_{\Sigma} (m\,\cdot\,\nu) \vert u_t\vert^2\; d\Gamma\; dt -\frac{1}{2}\int_{\Sigma} (m\,\cdot\,\nu) \vert\Delta u\vert^2\; d\Gamma\;dt-\int_{\Sigma} \partial_\nu \Delta u (m\,\cdot\,\nabla u)\;  d\Gamma\; dt\nonumber\\
	&&	+\int_{\Sigma} \partial_\nu u\; \Delta u\;   d\Gamma\; dt + \int_{\Sigma}\Delta u (m\,\cdot\,\partial_\nu \nabla u)\;  d\Gamma\; dt,\label{10}
	\end{eqnarray}}
	then we have
	\begin{eqnarray}
	&&		\left[ \int_\Omega u_t \; \psi m\,\cdot\,\nabla udx\;dy \right]^T_0 +\frac{1}{2}\int_Q div(\psi m)  \vert u_t\vert^2\; dx\;dy\;dt\nonumber\\
	&& +2\sum_{j,k} \int_Q \frac{\partial (\psi m)_k}{\partial x_j} \Delta u \frac{\partial^2 u}{\partial x_k \partial x_j}\;dx\;dy\; dt		+\int_Q (\Delta (\psi m)\,\cdot\,\nabla u) \; \Delta u \;dx\;dy\;dt\nonumber\\
	&&-\frac{1}{2} \int_Q  div(\psi m) \vert\Delta u\vert^2\; dx\;dy\;dt\nonumber\\
	&&		=\left[\alpha \int_\Omega uu_t\,dx\;dy +\int_\Omega (u_t m\,\cdot\,\nabla u)dx\;dy\right]^T_0 +(1-\alpha)\int_Q   \vert u_t\vert^2\; dx\;dy\;dt  \nonumber\\
	&&	+\;\alpha\int_Q \Vert u\Vert^2_{H^2_*}\;dx\;dy\;dt+\int_Q \vert\Delta u\vert^2\;dx\;dy\;dt.\label{11} 
	\end{eqnarray}
	
	Note that $\psi=0$ in $\Omega\setminus \omega$. We deduce that
	\begin{eqnarray*}
		&&\frac{1}{2} \int_Q  div(\psi m) \vert u_t\vert^2\;dx\;dy\;dt=\frac{1}{2} \int_0^T\int_\omega  div(\psi m) \vert u_t\vert^2\;dx\;dy\;dt\\
		&&\leq C_1\int_0^T\int_\omega   \vert u_t\vert^2\;dx\;dy\;dt= C_1\int_0^T\int_\Omega \chi_\omega \vert u_t\vert^2\;dx\;dy\;dt,
	\end{eqnarray*}
	where $C_1>0$.\\
	
	Also we have
	\begin{eqnarray*}
		-\int_Q div(\psi m) \vert \Delta u\vert^2 \;dx\;dy\;dt&\leq& C_2\int_0^T\int_{supp(\psi)} \vert \Delta u\vert^2 \;dx\;dy\;dt\\&\leq& \hat{C}_2\int_0^T\int_{supp(\psi)}  F(u,u)\;dx\;dy\;dt,
	\end{eqnarray*}
	where $C_2, \hat{C}_2>0.$\\
	
	\medskip

	{\text {\bf Estimation of the term}}\;\; $\displaystyle{\sum_{j,k}} \displaystyle{\int}_Q \frac{\partial (\psi m)_k}{\partial x_j} \Delta u \frac{\partial^2 u}{\partial x_k \partial x_j}\;dx\;dy\; dt\,.$\\
	Indeed,
	\begin{eqnarray*}
	&&	\sum_{j,k} \int_Q \frac{\partial (\psi m)_k}{\partial x_j} \Delta u \frac{\partial^2 u}{\partial x_k \partial x_j}\;dx\;dy\; dt\\
	&&\leq \sum_{j,k}\int_0^T \int_{supp(\psi)} \vert \Delta u \vert\;\; \vert \frac{\partial^2 u}{\partial x_k \partial x_j}\vert \;dx\;dy\; dt\\
&&		\leq \int_0^T \left(\int_{supp(\psi)}\vert \Delta u\vert^2\;dx\;dy\right)^{\frac{1}{2}}\left(\int_{supp(\psi)}\left(\sum_{j,k}\vert  \frac{\partial^2 u}{\partial x_k \partial x_j}\vert\right)^2\;dx\;dy\right)^{\frac{1}{2}}\;dt
\\&&	\leq C_{\varepsilon^\prime} \int_0^T \int_{supp(\psi)} \vert \Delta u\vert^2\;dx\;dy\;dt +4\varepsilon^\prime\int_0^T\int_{supp(\psi)}\sum_{j,k}\vert  \frac{\partial^2 u}{\partial x_k \partial x_j}\vert^2\;dx\;dy\;dt\\
&& \leq \hat{C}_{\varepsilon^\prime} \displaystyle{\int_0^T} \displaystyle{\int_{supp(\psi)}} F(u,u)\;dx\;dy\;dt + C_3 \varepsilon^\prime \int_0^T E(t)\;dt,
	\end{eqnarray*} \\
where $C_{\varepsilon^\prime}, \varepsilon^\prime, \hat{C}_{\varepsilon^\prime}$ and $C_3$ are positive constants.\\
	
	{\text {\bf Estimation of the term}}\;\;$\displaystyle{\int_Q} (\Delta (\psi m)\,\cdot\,\nabla u) \; \Delta u \;dx\;dy\;dt\,.$\\
	
	We observe that
	\begin{align*}
		&\int_Q (\Delta (\psi m)\,\cdot\,\nabla u) \; \Delta u \;dx\;dy\;dt\\&\leq \tilde{C}_{1, \varepsilon^\prime} \int_0^T \int_{supp(\psi)} \vert \Delta u\vert^2\;dx\;dy\;dt+\varepsilon^\prime\int_0^T \int_{supp(\psi)} \vert \nabla u\vert^2\;dx\;dy\;dt\\
		&\leq C_4 \varepsilon^\prime \displaystyle{\int_0^T} E(t)\;dt  + \hat{C}_{2,\varepsilon^{\prime}}\displaystyle{\int_0^T}\displaystyle{\int_{supp(\psi)}} F(u, u)\;dx\;dy\;dt.
	\end{align*}\\

	Combining all the above estimates and choosing $\varepsilon^\prime>0$ small enough we obtain
{\small	\begin{eqnarray}
\hspace*{-.3cm}	&&	\nonumber	\int_0^T E(t)\;dt
\nonumber\leq C_5\left( \left\vert \left[\alpha \int_\Omega uu_t\;dx\;dy +\int_\Omega (u_t m\,\cdot\,\nabla u)\;dx\;dy\right]^T_0 \right\vert\right.\\
&&\left.  + \left[ \int_\Omega u_t \; (\psi m\,\cdot\,\nabla u)\;dx\;dy \right]^T_0\right)	+C_6\left(\int_0^T\int_{supp(\psi)} F(u, u)\;dx\;dy\;dt \right)\!\!.\label{12}
	\end{eqnarray}}
	
	It remains to estimate the term $\displaystyle{\int_0^T}\displaystyle{\int_{supp(\psi)}} F(u, u)\;dx\;dy\;dt$ in terms of the damping term.\\
	
	\underline{Step 2}\;\;
	
	Let $\eta:\mathbb{R}^2\to \mathbb{R}$ a smooth function to be determined later. Multiplying equation (\ref{1}) by $\eta u$ and performing integration by parts, yields\\\\
	\begin{eqnarray}
	&&	\left[ \displaystyle{\int_\Omega} u_t \; \eta  u\;dx\;dy \right]^T_0 -\displaystyle{\int_Q} \eta \vert u_t\vert^2\;dx\;dy\;dt
	+\displaystyle{\int_Q} \Big( \Delta u \Delta \eta u +\vert \Delta u\vert^2 \eta +2 \Delta u \eta_x u_x \nonumber\\
	&& +2\Delta u \eta_y u_y 		+ 2(1-\sigma) u_{xy}\eta_{xy} u +   2(1-\sigma) u_{xy}\eta_{x} u_y +2(1-\sigma) u_{xy}\eta_{y} u_x \nonumber\\
	&&	 +2(1-\sigma) \eta u_{xy}^2	- (1-\sigma) u_{xx}\eta_{yy} u - 2(1-\sigma) u_{xx}\eta_{y} u_y - (1-\sigma) \eta u_{xx}u_{yy}\nonumber\\
	&&	- (1-\sigma) u_{yy}\eta_{xx} u - 2(1-\sigma) u_{yy}\eta_{x} u_x - (1-\sigma) \eta u_{xx} u_{yy}\Big) \;dx\;dy\;dt=0,\label{13}
	\end{eqnarray}
	that is, \\
	\begin{eqnarray}
	&&	\displaystyle{\int_Q} \eta F(u, u)\;dx\;dy\;dt=-\left[ \displaystyle{\int_\Omega} u_t \; \eta  u\;dx\;dy \right]^T_0 +\displaystyle{\int_Q} \eta \vert u_t\vert^2\;dx\;dy\;dt \nonumber\\
	&&	- \displaystyle{\int_Q} \Big( \Delta u \Delta \eta u  +2 \Delta u \eta_x u_x +2\Delta u \eta_y u_y + 2(1-\sigma) u_{xy}\eta_{xy} u \nonumber\\
	&&		+ 2(1-\sigma) u_{xy}\eta_{x} u_y +2(1-\sigma) u_{xy}\eta_{y} u_x- (1-\sigma) u_{xx}\eta_{yy} u  \nonumber\\
	&&		 - 2(1-\sigma) u_{xx}\eta_{y} u_y - (1-\sigma) u_{yy}\eta_{xx} u - 2(1-\sigma) u_{yy}\eta_{x} u_x \Big) \;dx\;dy\;dt\,.\label{14}
	\end{eqnarray}
	
	Let us define $\eta$.
	Consider Figure \ref{fig 4} and set the compact subsets $\hat U=\Omega \backslash \Omega_{2\varepsilon}$ and $\hat V=\Omega \backslash \omega$ of $\Omega$.
	Lemma \ref{existencia suave intermediario} states that there exist open subsets $U$ and $V$ of $\Omega$ with smooth boundaries and disjoint closures such that $\hat U \subset U$ and $\hat V \subset V$. For the sake of convenience we define $A=\Omega \backslash \bar U$ and $B=V$ (see Figure \ref{fig 4}).
	
	Now we define the smooth function $\eta$ according to Theorem \ref{vizinhanca tubular em fechados}. We have that
	\begin{eqnarray*}
		\eta(x)=\left\{
		\begin{array}{lcr}
			1\;\; \text{if}\;\;x\in\Omega\backslash A,\\\\
			0\;\;\text{if}\;x \in \bar B, \\\\
			\text{in the interval } (0,1) \; \; \text{in } A\backslash \bar B\;,
			
		\end{array}
		\right.
	\end{eqnarray*}
	and the behavior of $\eta$ in a tubular neighborhood of $\partial (A\backslash \bar B)$ contained in the closure of $A\backslash \bar B$ is given by $\text{dist}^4(x,\partial B)$, near $\partial B$ and $1-\text{dist}^4(x,\partial A)$ near $\partial A$ (for a complete definition, see Theorem \ref{vizinhanca tubular em fechados}).\\
	
	Then $\frac{\vert \Delta \eta\vert^2}{\eta}, \frac{\vert \eta_x\vert^2}{\eta}, \frac{\vert \eta_y\vert^2}{\eta}, \frac{\vert \eta_{xy}\vert^2}{\eta},  \frac{\vert \eta_{xx}\vert^2}{\eta}$ and $\frac{\vert \eta_{yy}\vert^2}{\eta}$ are bounded in $\Omega\backslash \bar B$ due to Theorem \ref{vizinhanca tubular em fechados}, Remark \ref{laplacian and other derivatives} and Remark \ref{observacao derivada primeira}.
	Thus, $supp(\eta)\subset \Omega\backslash \bar B$ and $\eta \geq 0$. Combining these facts, it follows that
	
	$$\displaystyle{\int_Q} \eta \vert u_t\vert^2\;dx\;dy\;dt\leq C_7 \int_0^T \int_\Omega \chi_\omega \vert u_t\vert^2\;dx\;dy\;dt.$$
	
	Let us estimate the term $ \displaystyle{\int_Q}  \Delta u \Delta \eta u\;dx\;dy\;dt$.\\
	We have
	\begin{eqnarray*}
		&&\int_0^T\int_\omega \Delta u \Delta \eta u\;dx\;dy\;dt\leq  \int_0^T\int_{\Omega\backslash \bar B} \sqrt{\eta}\;\vert \Delta u\vert\; \frac{\vert \Delta \eta \vert}{\sqrt{\eta}}\;\vert u\vert \;dx\;dy\;dt\,\\
		&&\leq \varepsilon^\prime \int_0^T \int_{\Omega\backslash \bar B} \eta F(u, u)\;dx\;dy\;dt + \hat{C}_{6, \varepsilon^\prime} \int_0^T \int_{\Omega\backslash \bar B} \frac{\vert \Delta \eta\vert^2}{\eta}\; \vert u\vert^2\;dx\;dy\;dt.
	\end{eqnarray*}
	
	From the above, we have that $\frac{\vert \Delta \eta\vert^2}{\eta}\in L^{\infty}(\Omega\backslash \bar B)$. Taking the other terms which come from (\ref{14}) into account we also have, by construction that $\frac{\vert \eta_x\vert^2}{\eta}, \frac{\vert \eta_y\vert^2}{\eta}, \frac{\vert \eta_{xy}\vert^2}{\eta},  \frac{\vert \eta_{xx}\vert^2}{\eta}$ and $\frac{\vert \eta_{yy}\vert^2}{\eta}$ are bounded in $\Omega\backslash \bar B$.\\

	So, having in mind that $\eta=1$ on  $supp(\psi)$ and  $\eta\leq1 $ on $\Omega\backslash \bar B$, we infer

	\begin{figure}[h!]
\centering
\includegraphics[width=1\linewidth]{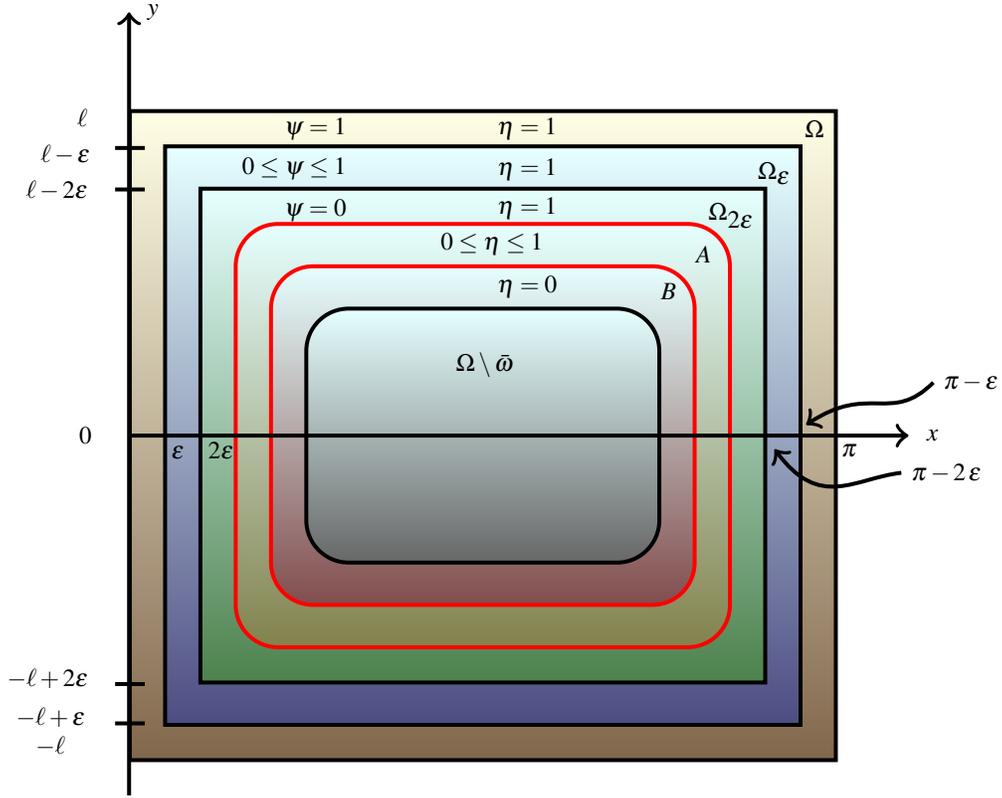}
\caption{Smooth function $ \eta. $}
\label{fig 4}
\end{figure}

{\small 	\begin{eqnarray*}
		\int_0^T\int_{supp(\psi)}\!\!\!\!\!\!  F(u, u)\;dx\;dy\;dt= \int_0^T\int_{supp(\psi)}\!\!\!\!\!\!\!\!\! \eta F(u, u)\;dx\;dy\;dt\leq\int_0^T\int_{\Omega\backslash \bar B} \eta F(u, u)\;dx\;dy\;dt.
	\end{eqnarray*}}
	
	We obtain from (\ref{12}), (\ref{14}) and the above similar estimations
	\begin{eqnarray}
	&&		\int_0^T E(t)\;dt\leq C_8\Big( \left\vert \left[\alpha \int_\Omega uu_t \;dx\;dy +\int_\Omega (u_t m\,\cdot\,\nabla u)\;dx\;dy\right]^T_0 \right\vert \nonumber\\
	&& + \left[ \int_\Omega u_t \; (\psi m\,\cdot\,\nabla u)\;dx\;dy \right]^T_0		+\left\vert\left[ \displaystyle{\int_\Omega} u_t \; \eta  u\;dx\;dy \right]^T_0\right\vert +\int_Q \chi_\omega \vert u_t\vert^2\;dx\;dy\;dt\nonumber\\
	&&	+\int_Q  \vert u\vert^2\;dx\;dy\;dt+\int_Q  \vert u_x\vert^2\;dx\;dy\;dt+\int_Q  \vert u_y\vert^2\;dx\;dy\;dt\Big).\label{15}
	\end{eqnarray}

	The last step is to prove that
	\begin{eqnarray}
	&&	\int_Q  \vert u\vert^2\;dx\;dy\;dt+\int_Q  \vert u_x\vert^2\;dx\;dy\;dt+\int_Q  \vert u_y\vert^2\;dx\;dy\;dt\label{16}\\
	&&	\leq C \int_Q \chi_\omega \vert u_t\vert^2\;dx\;dy\;dt,\nonumber
	\end{eqnarray}
	for some positive constant $C$.\\
	\underline{Step 3}\;\;\
	We argue by contradiction. Let us suppose that (\ref{16}) is not satisfied and let $(u_0^k, u_1^k)_k $ be a sequence of initial data where the corresponding solutions $(u^k)_k$ of (\ref{1}), with $E_k(0)$ assumed uniformly bounded in $k$, satisfy
	$$\lim_{k\to+\infty} \frac{\displaystyle{\int_Q}  \left(\vert u^k\vert^2 +\vert u^k_x\vert^2+\vert u^k_y\vert^2\right)\;dx\;dy\;dt}{\displaystyle{\int_Q} \chi_\omega\vert u^k_t\vert^2\;dx\;dy\;dt }=+\infty,$$
	that is
	\begin{equation}\label{17}
	\lim_{k\to+\infty} \frac{  \displaystyle{\int_Q} \chi_\omega \vert u^k_t\vert^2\;dx\;dy\;dt}{\displaystyle{\int_Q}  \left(\vert u^k\vert^2 +\vert u^k_x\vert^2+\vert u^k_y\vert^2\right)\;dx\;dy\;dt}=0\,.
	\end{equation}
	Since $E_k(t)\leq E_k(0)\leq L$, with $L>0$ independent of $k$, we obtain a subsequence, still denoted by $(u_k)_k$,  which satisfies the convergence
	\begin{equation}\label{18}
	u^k\rightharpoonup u \;\;{\text{weakly\; star\; in}}\;\; L^{\infty}(0, T; H^2_*(\Omega))\,.
	\end{equation}
	\begin{equation}\label{19}
	u_t^k\rightharpoonup u_t \;\;{\text{weakly\; star\; in}}\;\; L^{\infty}(0, T; L^2(\Omega))\,.
	\end{equation}
	Thanks to the compact embedding $H^2_*(\Omega)\subset L^2(\Omega)$ and  $H^2_*(\Omega)\subset H^1(\Omega)$, the obtain
	\begin{equation}\label{20}
	u^k\to u \;\;{\text{strongly\; in}}\;\; L^{2}(0, T; L^2(\Omega))\,.
	\end{equation}
	\begin{equation}\label{21}
	u_x^k \to u_x \;\;{\text{strongly\; in}}\;\; L^{2}(0, T; L^2(\Omega))\,.
	\end{equation}
	\begin{equation}\label{22}
	u_y^k \to u_y \;\;{\text{strongly\; in}}\;\; L^{2}(0, T; L^2(\Omega))\,.
	\end{equation}
	At this point we will divide our proof into two cases: $u\neq 0$ and $u=0$ .\\
	
	{\bf Case(I):} $u\neq 0$:\\
	We also observe that  form (\ref{17}), (\ref{20}), (\ref{21}) and (\ref{22}) we have
	\begin{equation}\label{23}
	\lim_{k\to +\infty} \int_0^T\int_\Omega \chi_\omega \vert u_t^k\vert^2\;dx\;dy\;dt=0\,.
	\end{equation}
	Passing to the limit in the equation, when $k\to +\infty$, we get

{	\small \begin{equation}\label{24}
	\begin{cases}
	u_{tt}(x,y,t)+\Delta^2 u(x,y,t)=0, &\mbox{ in } \;\Omega \times (0,+\infty), \\\\
	
	u(0,y,t)=u_{xx}(0,y,t)=u(\pi,y,t)=u_{xx}(\pi,y,t)=0, &(y,t)\in(-l,l)\times (0,+\infty), \\\\
	
	u_{yy}(x,\pm l,t)+\sigma u_{xx}(x,\pm l, t)=0, &(x,t)\in (0,\pi)\times (0, +\infty), \\\\
	
	u_{yyy}(x,\pm l,t)+(2-\sigma)u_{xxy}(x,\pm l, t)=0, &(x,t)\in(0,\pi)\times(0,+\infty),\\\\
	
	u_t(x,y,t)=0, &\mbox{ in } \;\omega\times (0, +\infty),
		\end{cases}
		\end{equation}
	}
	and for $u_t=v$, we obtain in the distributional sense
	\begin{eqnarray}\label{25}
	\left\{
	\begin{array}{lcr}
	v_{tt}(x,y,t)+\Delta^2 v(x,y,t)=0, &\mbox{ in } \;\Omega \times (0,+\infty),& \\\\
	
	v(x,y,t)=0, &\mbox{ in } \;\omega\times (0,+\infty),&
	\end{array}
	\right.
	\end{eqnarray}
	From Holmgren's uniqueness theorem, we deduce that $v=u_t=0\;in \;\Omega$. Then we obtain:
	\begin{eqnarray}\label{26}
	\left\{
	\begin{array}{lcr}
	\Delta^2 u(x,y)=0, &\mbox{ in } \;\Omega,&\\\\
	
	u(0,y)=u_{xx}(0,y)=u(\pi,y)=u_{xx}(\pi,y)=0, &y\in(-l,l),& \\\\
	
	u_{yy}(x,\pm l)+\sigma u_{xx}(x,\pm l)=0,  &x\in (0,\pi),&\\\\
	
	u_{yyy}(x,\pm l)+(2-\sigma)u_{xxy}(x,\pm l)=0, &x\in(0,\pi),&
	
	\end{array}
	\right.
	\end{eqnarray}
	By using \cite{Gazzola1} (Theorem 3.2), we conclude that $u=0$. So, we obtain a contradiction.\\
	
	{\bf Case (II):} $u=0$ \\
	Define $$c_k=\left[\displaystyle{\int_Q}  \left(\vert u^k\vert^2 +\vert u^k_x\vert^2+\vert u^k_y\vert^2\right)\;dx\;dy\;dt\right]^{\frac{1}{2}},$$
	and
	$$\overline{u}^k=\frac{u^k}{c_k}.$$
	We obtain
	\begin{equation}\label{27}
	\displaystyle{\int_Q}  \left(\vert \overline{u}^k\vert^2 +\vert \overline{u}^k_x\vert^2+\vert \overline{u}^k_y\vert^2\right)\;dx\;dy\;dt=1\,.
	\end{equation}
	We set
	$$\overline{E}_k(t)=\frac{1}{2}\left(\int_\Omega \vert \overline{u}_t^k\vert^2\;dx\;dy + \Vert \overline{u}^k_t\Vert^2_{H^2_*(\Omega)}\right)\,.$$
	We deduce that
	\begin{equation}\label{28}
	\overline{E}_k=\frac{E_k}{c_k^2}\,.
	\end{equation}
	On the other hand,
	$$\left\vert \left[\alpha \int_\Omega u u_t\;dx\;dy +\int_\Omega (u_t m\,\cdot\,\nabla u)\;dx\;dy\right]^T_0 \right\vert$$
	$$\leq C_{10} \left[\alpha \int_\Omega \vert u\vert^2\;dx\;dy +\int_\Omega \vert u_t\vert^2\;dx\;dy +\int_\Omega \vert \nabla u\vert^2\;dx\;dy\right]_0^T$$
	\begin{equation}\label{29}
	\leq C_{11}(E(T) +E(0))=2C_{11}E(T).
	\end{equation}
	
Analogously, we prove that
	\begin{equation}\label{30}
	\left[ \int_\Omega u_t \; \psi m\,\cdot\,\nabla u\;dx\;dy \right]^T_0 \leq C_{12}(E(T)+ \int_Q \chi_\omega\vert u_t\vert^2\;dx\;dy\;dt),
	\end{equation}
	and
	\begin{equation}\label {31}
	\left\vert\left[ \displaystyle{\int_\Omega} u_t \; \eta  u\;dx\;dy \right]^T_0\right\vert \leq C_{13}(E(T)+ \int_Q \chi_\omega\vert u_t\vert^2\;dx\;dy\;dt).
	\end{equation}
	By the use of (\ref{15}), (\ref{29}), (\ref{30}), (\ref{31}) and the obvious equality
	$$TE(T) =\int_0^T E(t)\;dt,$$
	we obtain, for $T$ large enough, the existence of a constant $C_{14}>0$ such that
	\begin{equation}\label{32}
	E(T)\leq C_{14} \left(\int_Q \chi_\omega\vert u_t\vert^2\;dx\;dy\;dt + \int_Q \left( \vert u\vert^2 +\vert u_x\vert^2+\vert u_y\vert^2\right)\;dx\;dy\;dt\right)
	\end{equation}
	and then,
	$$E(t)\leq E(0)\leq  \tilde{C}_{14} \left(\int_Q \chi_\omega\vert u_t\vert^2\;dx\;dy\;dt + \int_Q \left( \vert u\vert^2 +\vert u_x\vert^2+\vert u_y\vert^2\right)\;dx\;dy\;dt\right)\,.$$
	The last inequality and (\ref{28}) give us
	$$\overline{E}_k(t)\leq \tilde{C}_{14} \left(  \frac{  \displaystyle{\int_Q} \chi_\omega \vert u^k_t\vert^2\;dx\;dy\;dt}{\displaystyle{\int_Q}  \left(\vert u^k\vert^2 +\vert u^k_x\vert^2+\vert u^k_y\vert^2\right)\;dx\;dy\;dt}+1   \right)\,.$$
	From (\ref{17}), we conclude that there exist a positive constant $\overline{L}$ such that
	$$\overline{E}_k(t)\leq \overline{L}, \forall\;t\in[0, T],\;\;\forall\;k\in\mathbb{N},$$
	and consequently we have
	\begin{equation}\label{33}
	\overline{u}^k\to \overline{u} \;\;{\text{strongly\; in}}\;\; L^{2}(0, T; L^2(\Omega)),
	\end{equation}
	and
	$$\overline{u}_x^k \to \overline{u}_x \;\;{\text{strongly\; in}}\;\; L^{2}(0, T; L^2(\Omega))$$
	$$\overline{u}_y^k \to \overline{u}_y \;\;{\text{strongly\; in}}\;\; L^{2}(0, T; L^2(\Omega))$$
	Now, it follows from (\ref{23}) that $\displaystyle \lim_{k\to +\infty}\int_0^T\int_\Omega \chi_\omega \vert \overline{u}^k_t\vert^2\;dx\;dy=0$.\\
	In addition, $\overline{u}^k$ satisfies the equation
	$$\overline{u}^k_{tt} + \Delta^2 \overline{u}^k =0\,.$$
	Passing to the limit, when $k\to +\infty$, and taking into account the above convergence, we obtain
{\small	\begin{equation*}
		\begin{cases}
			\overline{u}_{tt}(x,y,t)+\Delta^2 \overline{u}(x,y,t)=0, &\mbox{ in } \;\Omega \times (0,+\infty),\\\\
			
			\overline{u}(0,y,t)=\overline{u}_{xx}(0,y,t)=\overline{u}(\pi,y,t)=\overline{u}_{xx}(\pi,y,t)=0, &(y,t)\in(-l,l)\times (0,+\infty), \\\\
						\overline{u}_{yy}(x,\pm l,t)+\sigma \overline{u}_{xx}(x,\pm l, t)=0, &(x,t)\in (0,\pi)\times (0, +\infty), \\\\			\overline{u}_{yyy}(x,\pm l,t)+(2-\sigma)\overline{u}_{xxy}(x,\pm l, t)=0, &(x,t)\in(0,\pi)\times(0,+\infty),\\\\
						\overline{u}_t(x,y,t)=0, &\mbox{ in } \;\omega\times (0, +\infty),	
		\end{cases}
	\end{equation*}}
	
	and for $\overline{u}_t=\overline{v}$, we obtain in the distributional sense
	\begin{eqnarray*}
		\left\{
		\begin{array}{lcr}
			\overline{v}_{tt}(x,y,t)+\Delta^2 \overline{v}(x,y,t)=0, &\mbox{ in } \;\Omega \times (0,+\infty),& \\\\
			
			\overline{v}(x,y,t)=0, &\mbox{ in } \;\omega\times (0,+\infty),&
		\end{array}
		\right.
	\end{eqnarray*}
	Applying again Holmgren's uniqueness theorem, we deduce that $\overline{v}=\overline{u}_t=0\;in \;\Omega$. Then we obtain:
	\begin{eqnarray*}
		\left\{
		\begin{array}{lcr}
			\Delta^2 \overline{u}(x,y)=0, &\mbox{ in } \;\Omega,& \\\\
			
			\overline{u}(0,y)=\overline{u}_{xx}(0,y)=\overline{u}(\pi,y)=\overline{u}_{xx}(\pi,y)=0, &y\in(-l,l),& \\\\
			
			\overline{u}_{yy}(x,\pm l)+\sigma \overline{u}_{xx}(x,\pm l)=0, &x\in (0,\pi),& \\\\
			
			\overline{u}_{yyy}(x,\pm l)+(2-\sigma)\overline{u}_{xxy}(x,\pm l)=0, &x\in(0,\pi),&
			
		\end{array}
		\right.
	\end{eqnarray*}
	and consequently  $\overline{u}=0$, which is   a contradiction in view of (\ref{27}) and (\ref{33}).
$ \qedsymbol $

\section{The Nonlinear Model}

\subsection{Wellposedness}
{\color{black}
To classify the growth of the nonlinear feedbacks we introduce the notion of the polynomial order at infinity.
\begin{definition}[Order at infinity of a nonlinear map]\label{def:order}
A monotone increasing map $f:\mathrm{R}\to\mathbb{R}$, $f(0)=0$, is of  the order $r\dfn \cO(f)\geq 0$ at infinity, if there exists $c > 0$ such that
\begin{equation}\label{i:order}
\begin{split}
 |s|^{r+1} \sim   f(s)s&\quad \text{whenever}\quad |s|\geq c.
\end{split}
\end{equation}
When the order $r$ exceeds, falls below, or equals $1$ we say the map $f$ is respectively: superlinear, sublinear, or linearly bounded at infinity.
\end{definition}

Based on the above definition, the function $g$ is assumed to be continuous and monotonic increasing such that
\begin{equation}\label{g}
\left\{
\begin{aligned}
&g(s)s >0 \hbox{ for all } s\ne 0,\\
&\alpha_1 |s|^{r+1} \leq g(s)s \leq \alpha_2 |s|^{r+1}\hbox{ for all } |s| \geq 1,
\end{aligned}
\right.
\end{equation}
for some positive constants $\alpha_1,\alpha_2$.

\begin{asmp}[Regularity for sub- and superlinear feedbacks at infinity]\label{as:regularity}
 This assumption  is imposed only when $g$ is not linearly bounded at infinity:

 \begin{itemize}
 	\item    If $\cO(g) \neq 1$  assume  $u_t \in L^\infty\big (\bbR_+; L^{p_{0}}(\cM)\big)$,	where  $p_{0}> 2\max\{1,\cO(g)\}$.
	
\end{itemize}
 \end{asmp}	
\begin{remark}\label{remark 2}
Note that since the system is monotone dissipative,  the regularity Assumption \ref{as:regularity} can be satisfied to a certain extent by starting with smooth initial data. Thus, if a solution is \textbf{regular} (as described below) then, $u_{t}\in L^{\infty}(\mathbb{R}_{+};H_*^2(\Omega))$, hence $u_{t}\in L^{\infty}(\mathbb{R}_{+};L^{p_{0}}(\Omega))$ for any $p_{0}<\infty$, because $\dim \Omega=2$. Consequently, when $(u^{0}, u^{1})$ belong to the domain $D(A)$ of the evolution generator (as defined in section 2) there is no restriction on  $\cO(g)$. Remember that we shall work with regular solutions and for standard density arguments the decay rate estimates remain valid for weak solutions as well.
\end{remark}
}

Inspired in \cite{Alabau}, \cite{Alabau2}, \cite{Alabau3}, \cite{Cavalcanti0} and \cite{Lasiecka-Tataru}, let $h$ be a\ concave, strictly
increasing function, with $ h\left( 0\right) =0$, and such that
\begin{equation}
h\left( s\,g(s)\right) \geq s^{2}+ g^{2}(s),\text{\ for } |s|<1.
\label{h}
\end{equation}

Problem (\ref{main problem}) can be written
\begin{eqnarray*}
	\left\{
	\begin{aligned}
		& U_t + \mathcal{A} U =G,\\
		& U(0)=U_0,
	\end{aligned}
	\right.
\end{eqnarray*}
where
{\small\begin{eqnarray*}
	U=\left(
	\begin{aligned}
		&u\\
		&v
	\end{aligned}
	\right);\,
	\mathcal{A} U:=\left(
	\begin{aligned}
		& -v\\
		& \Delta^2u + a(\cdot) g(v)
	\end{aligned}
	\right);\,
	G(U)=\left(
	\begin{aligned}
		&~0\\
		&-\phi(u)u_{xx}
	\end{aligned}
	\right)
	\hbox{ and }\,
	U_0=\left(
	\begin{aligned}
		&u_0\\
		&v_0
	\end{aligned}
	\right),
\end{eqnarray*}}
where  $D(\mathcal{A})=D(A)$ has been defined in the previous section. It is not difficult to prove by using standard nonlinear semigroup theory that $\mathcal{A}$ is maximal monotone operator in $\mathcal{H}$ (see, for instance, \cite{Cavalcanti-DiasSilva-Domingos}). Thus, in order to prove that problem (\ref{main problem}) is wellposed it is sufficient to prove that:

\begin{lemma}
	$G$  is locally Lipschitz in $\mathcal{H}$.
\end{lemma}
\noindent{\bf Proof:}
	We need to prove that given $R>0$ there exists $C(R)>0$ such that
	\begin{eqnarray}\label{I}
	||G(U) - G(V)||_{\mathcal{H}} \leq C(R) \, || U-V||_{\mathcal{H}},~\hbox{ provided that }||U||_{\mathcal{H}}, ||V||_{\mathcal{H}}\leq R.
	\end{eqnarray}
	
	One has
{\small	\begin{eqnarray}\label{II}
	&&||G(U) - G(V)||_{\mathcal{H}}^2 \\
	&&= \int_{\Omega} \left|\phi(u)u_{xx} - \phi(\tilde u)\tilde {u}_{xx} \right|^2\,dx\nonumber\\
	&&= \int_{\Omega} \left|-P(u_{xx} -\tilde u_{xx}) +S\left[\left(\int_\Omega u_x^2\,dx\right)u_{xx} -\left(\int_\Omega \tilde u_x^2\,dx\right)\tilde u_{xx} \right]  \right|^2 \,dx\nonumber\\
	&&\leq L_1  \int_{\Omega} |u_{xx} -\tilde u_{xx}|^2\, dx + L_1 \int_\Omega \left|\left(\int_\Omega u_x^2\,dx\right)u_{xx} -\left(\int_\Omega \tilde u_x^2\,dx\right)\tilde u_{xx}  \right|^2\,dx,\nonumber
	\end{eqnarray}}
	where $L_1$ is a positive constant.
	
	However,
{\small	\begin{eqnarray}\label{III}
	&& \int_\Omega \left|\left(\int_\Omega u_x^2\,dx\right)u_{xx} -\left(\int_\Omega \tilde u_x^2\,dx\right)\tilde u_{xx}  \right|^2\,dx\\
	&& =\int_\Omega \left| \left(\int_\Omega u_x^2\,dx\right)(u_{xx}- \tilde u_{xx}) + \tilde u_{xx}\left(\int_\Omega u_x^2\,dx - \int_\Omega \tilde u_x^2\,dx
	\right)  \right|^2\,dx\nonumber\\
	&&\leq L_2 \left(\int_\Omega u_x^2\,dx\right)^2 \int_\Omega |u_{xx}- \tilde u_{xx}|^2 \,dx + L_2 \int_\Omega |\tilde u_{xx}|^2\,dx\, \left(\int_\Omega u_x^2\,dx - \int_\Omega \tilde u_x^2\,dx\right)^2\nonumber\\
	&& \leq L_3(R)\int_\Omega |u_{xx}- \tilde u_{xx}|^2 \,dx + L_3(R) \left(\int_\Omega (u_x^2 -  \tilde u_x^2)\,dx\right)^2\nonumber\\
	&&\leq L_3(R)\int_\Omega |u_{xx}- \tilde u_{xx}|^2 \,dx + L_4(R) \int_\Omega |u_x -  \tilde u_x|^2\,dx,\nonumber
	\end{eqnarray}}
	where $L_2$, $L_3=L_3(R)$ and $L_4=L_4(R)$ are positive constants.

	Combining  (\ref{II}) and (\ref{III}) yields $||G(U) - G(V)||_{\mathcal{H}} \leq C(R) \, || U-V||_{\mathcal{H}}$ as we desire to prove.
$ \qedsymbol $

\medskip

Thus, for $U_0 \in \mathcal{H}$ given, then according to standard semigroup properties problem (\ref{main problem}) possesses a unique solution $U\in C([0,\infty);\mathcal{H})$. In addition, if $U_0  \in D(A)$, then problem (\ref{main problem}) has a unique regular solution $U\in C([0,\infty);D(A)) \cap C^1([0,\infty);\mathcal{H})$.
\medskip

\subsection{Uniform Decay Rate Estimates}

The energy associated to problem (\ref{main problem}) is now defined by
{\small\begin{eqnarray}\label{energy}
E_u(t) = \underbrace{\frac12 ||u_t (t)||_{L^2(\Omega)}^2}_{\mathcal{K}_u(t)} + \underbrace{\frac12 ||u(t)||_{H^2_*(\Omega)}^2 - \frac{P}2 ||u_x(t) ||_{L^2(\Omega)}^2 + \frac{S}4 ||u_x(t)||_{L^2(\Omega)}^4}_{\mathcal{P}_u(t)},
\end{eqnarray}}
where $ t\geq 0. $ Here, $\mathcal{K}_u(t)$ and $\mathcal{P}_u(t)$ represent, respectively, the kinetic and the elastic potential energy of the model. Moreover, one has the identity of the energy
{\small \begin{eqnarray}\label{indentity of energy}
E_u(t_2) - E_u(t_1) = -\int_Q a(x,y) g(u_t(x,y,t))u_t(x,y,t)\, dx\,dy\,dt,
\end{eqnarray}}
 so that $ 0\leq t_1\leq t_2<+\infty, $ which shows that the energy is monotonic (non increasing).

We observe that when $P<0$, then $E_u(t) \geq 0$ for all $t\geq 0$. In elasticity this situation corresponds to a plate that has been stretched rather than compressed, which does not occur in actual bridges. So, when $P>0$, the most accurate case for bridges, the energy is no longer non negative, which plays an essential role in stabilization of distributed systems. To overcome this situation we will follow ideas from [\cite{Gazzola1'}, section 3]. Let us define
\begin{eqnarray*}
	H_*^1(\Omega)&:=&\{w\in H^1(\Omega): w=0 \text{  on } \{0,\pi\}\times (-l,l)\},\\
	C_*^\infty(\Omega) &:=& \{w\in C^\infty(\overline{\Omega}): \exists\, \varepsilon >0, w(x,y)=0 \hbox{ if }x\in  [0,\varepsilon] \cup [\pi-\varepsilon,\pi]\},
\end{eqnarray*}
which is a normed space when endowed with the Dirichlet norm
\begin{eqnarray}\label{Dirichlet norm}
||u||_{H_*^1(\Omega)}:= \left( \int_\Omega |\nabla u|^2\,dx\;dy\right)^{1/2}.
\end{eqnarray}

Then, we define $H^1_*(\Omega)$ as the completion of $C_*^\infty(\Omega)$ with respect to the norm $||\cdot||_{H^1_*(\Omega)}$. It is not difficult to prove the embedding $H_*^2(\Omega)\hookrightarrow H_*^1(\Omega)$ is compact and, further, that the optimal embedding constant is given by
\begin{eqnarray*}\label{Gamma1}
	\Lambda_1:= \min_{w\in H_*^2(\Omega) }\frac{||w||_{H_*^2(\Omega)}^2}{||w||_{H_*^1(\Omega)}^2},
\end{eqnarray*}
from what follows the Poincar\'e-type inequality
\begin{eqnarray}\label{poincare ineq}
||w||_{H_*^1(\Omega)}^2 \leq \Lambda_1^{-1} ||w||_{H_*^2(\Omega)}^2,~\hbox{ for all } w\in H_*^2(\Omega).
\end{eqnarray}

So, for all $u\in H_*^2(\Omega)$ and since
\begin{eqnarray*}
	||u_x||_{L^2(\Omega)}^2 \leq \int_\Omega |\nabla u|^2\, dx \leq \Lambda_1^{-1}\,||u||_{H_*^2(\Omega)}^2,
\end{eqnarray*}
yields
\begin{eqnarray*}
	-\frac{P}{2} ||u_x||_{L^2(\Omega)}^2 \geq -\frac{P}{2}\Lambda_1^{-1}\,||u||_{H_*^2(\Omega)}^2,
\end{eqnarray*}
and, therefore,
\begin{eqnarray*}
	\frac12 ||u||_{H_*^2(\Omega)}^2 -\frac{P}{2} ||u_x||_{L^2(\Omega)}^2 \geq \frac12 ||u||_{H_*^2(\Omega)}^2 \left(1-P \Lambda_1^{-1} \right).
\end{eqnarray*}

Thus, if $0\leq P \leq \Lambda_1$ from the last inequality we deduce that $\frac12 ||u||_{H_*^2(\Omega)}^2 -\frac{P}{2} ||u_x||_{L^2(\Omega)}^2 \geq 0$, and consequently $E_u(t) \geq 0$, which agrees with the assumption of Theorem 4 in \cite{Gazzola2}. We shall not work in the present paper with negative values of the energy because of the methodology used. It is worth mentioning that, if $E_u(t) <0$ necessarily $P> \Lambda_1$. However, under certain circumstances on the initial data it is possible to consider positive energy and $P$ not so small, namely, $\Lambda_1 < P \leq \Lambda_2$ as in Corollary 8 in \cite{Gazzola2}. It is important to observe that the physical meaningful values of prestressing are precisely when $P \leq \Lambda_2$ since otherwise the equilibrium positions of the plate may take unreasonable shapes as multiple buckling as mentioned in \cite{Gazzola2}. So, from now on we shall assume that $E(t)> 0$.

\medskip
The main result of this section reads as follows:
\begin{thm}\label{th2}
	For any $R > 0$ there exist constants $C$ and $T_0 > 0$, depending on $R$, such
	that, if $E_u((0))\leq R$, then
{\small	\begin{eqnarray}\label{main ineq}
	E_u(T) \leq C \int_0^T \int_\Omega a(x,y)\left[ |u_t(x,y,t)|^2 + |g(u_t(x,y,t))|^2\right]\,dx\,dy\,dt,~\forall T>T_0.
	\end{eqnarray}}
\end{thm}
\medskip
\noindent{\bf Proof:}
	It is enough to show that (\ref{main ineq}) holds for regular solutions, and to then use a density argument.\\
	
	\underline{Step 1}\;\; Having in mind we are just considering the nonlinear part of $\phi(u)$, namely, $S\displaystyle{\int_\Omega} u_x^2\,dx$, initially we note that problem (\ref{main problem}) can be written as a sum $u=v+w$ where $u$ and $v$, satisfy, respectively
{\small	\begin{equation}\label{v}
	\begin{cases}
	v_{tt}(x,y,t)+\Delta^2 v(x,y,t) =0, &\mbox{ in } \;\Omega \times (0,+\infty), \\\\
	
	v(0,y,t)=v_{xx}(0,y,t)=v(\pi,y,t)=v_{xx}(\pi,y,t)=0, &(y,t)\in(-l,l)\times (0,+\infty), \\\\
	
	v_{yy}(x,\pm l,t)+\sigma v_{xx}(x,\pm l, t)=0, &(x,t)\in (0,\pi)\times (0, +\infty), \\\\
	
	v_{yyy}(x,\pm l,t)+(2-\sigma)v_{xxy}(x,\pm l, t)=0, &(x,t)\in(0,\pi)\times(0,+\infty),\\\\
	
	v(x,y,0)=u_0(x,y),\; v_t(x,y,0)=u_1(x,y), &\mbox{ in } \;\Omega,
		\end{cases}
		\end{equation}}
	
	and
	
{\small	\begin{equation}\label{w}
	\begin{cases}
	w_{tt}(x,y,t)+\Delta^2 w(x,y,t) = -\phi(u) u_{xx}- a(x,y)g(u_t(x,y,t)), \quad\mbox{ in } \;\Omega \times (0,+\infty), \\\\
	
	w(0,y,t)=w_{xx}(0,y,t)=w(\pi,y,t)=w_{xx}(\pi,y,t)=0,\, (y,t)\in(-l,l)\times (0,+\infty), \\\\
	
	w_{yy}(x,\pm l,t)+\sigma w_{xx}(x,\pm l, t)=0,\qquad\qquad\qquad\qquad\quad\, (x,t)\in (0,\pi)\times (0, +\infty), \\\\
	
	w_{yyy}(x,\pm l,t)+(2-\sigma)w_{xxy}(x,\pm l, t)=0,\qquad\qquad\qquad\! (x,t)\in(0,\pi)\times(0,+\infty),\\\\
	
	w(x,y,0)= w_t(x,y,0)= 0, \qquad\qquad\qquad\qquad\quad\,\qquad\qquad\qquad\qquad\quad\,\qquad\!\! \mbox{ in } \;\Omega.
	
	\end{cases}
	\end{equation}}
	
	From now on we shall denote $E_u$,$E_v$ and $E_w$ the energies associated to $u,v$ and $w$. Then, once the map $t \mapsto E_u(t)$ is non increasing and exploiting the observability inequality associated to the linear problem $v$, we infer for $T_0>0$ large enough
	\begin{eqnarray}\label{eq1}
	E_u(T_0) &\leq& E_u(0) \\
	&=& \frac12 ||u_1||_2^2 + \frac12 ||u_0||_{H^2_\ast(\Omega)}^2 - \frac{P}2||u_{0,x}||_2^2 + \frac{S}4 ||u_{0,x}||_2^4\nonumber\\
	&\leq& L_1\left( ||u_1||_2^2 + ||u_0||_{H^2_\ast(\Omega)}^2\right)\nonumber\\
	&=& 2L_1 E_v(0) \nonumber\\
	&\leq& L_2 \int_0^{T_0} \int_\omega |v_t|^2\,dxdydt \nonumber\\
	&\leq& L_3 \int_0^{T_0} \int_\omega \left[|u_t|^2 + |w_t|^2\right]\,dxdydt \nonumber\\
	&\leq& L_4 \left(\int_0^{T_0} \int_\Omega a(x,y) |u_t|^2\,dxdydt + \int_0^{T_0} \int_\Omega |w_t|^2\,dxdydt\right),\nonumber
	\end{eqnarray}
	where $L_i$, $i=1,2,3,4$ are positive constants and the last inequality holds since $a(x,y) \geq a_0>0$ in $\omega$.\\
	
	We also mention  that to obtain the third line of (\ref{eq1}),  we used the fact that
	$$||u_{0,x}||_2^4\leq ||u_0||_{H^2_\ast(\Omega)}^4$$
	and
	$$||u_0||_{H^2_\ast(\Omega)}^2\leq 2E_u(0)\leq 2R.$$
	
	\underline{Step 2}\;\; Now, setting $f:= -\phi(u) u_{xx} - a(x,y)g(u_t(x,y,t)) \in L^2(0,T; L^2(\Omega))$ (see remark \ref{remark 2}) and $w(0)=w_t(0)=0$,  $\mathcal{L}:= L^\infty(0,T;H^2_\ast(\Omega))\times L^\infty(0,T;L^2(\Omega)) $ and $ \mathcal{H}:=H^2_\ast(\Omega)\times L^2(\Omega) \times L^2(0,T;L^2(\Omega)), $ it is known that  the linear map
{	\begin{align*}
	\{w(0),w_t(0),f\}\in \mathcal{H} \mapsto \{w,w_t\}\in \mathcal{L}
	\end{align*}}
	is continuous, we deduce
	\begin{eqnarray*}
		||w||_{L^\infty(0,T;H^2_\ast(\Omega))}^2 + ||w_t||_{L^\infty(0,T;L^2(\Omega))}^2 \leq C ||f||_{L^2(0,T;L^2(\Omega))}^2,
	\end{eqnarray*}
	from which follows that
	\begin{eqnarray}\label{eq2}
	||w_t||_{L^2(0,T;L^2(\Omega))}^2 \leq L_5 \left[||\phi(u) u_{xx}||_{L^2(0,T;L^2(\Omega))}^2 + ||a(\cdot) g(u_t)||_{L^2(0,T;L^2(\Omega))}^2  \right].
	\end{eqnarray}
	where $L_5$ is a positive constant.
	
	Combining (\ref{eq1}) and (\ref{eq2}) yields
{	\begin{align}\label{eq3}
	E_u(T_0)&\leq L_6 \left( \int_0^{T_0} \int_\Omega a(x,y) \left[|u_t|^2 + |g(u_t)|^2\right]\,dxdydt\right.\nonumber \\
  &	\left.+ \int_0^{T_0} \int_\Omega |\phi(u) u_{xx}|^2\,dxdydt\right).
	\end{align}}
	
	In the sequel let us analyse the term $I:= \displaystyle{\int_0^{T_0}} \displaystyle{\int_\Omega} |\phi(u) u_{xx}|^2\,dxdydt$. Remembering that we are considering $E_u(0) \leq R$, one has,
	\begin{eqnarray*}
		|I|&=& S^2 \int_0^{T_0} ||u_x(t)||_2^4 \int_\Omega|u_{xx}|^2\,dxdydt\\
		&=& S^2\int_0^{T_0} ||u_x(t)||_2^4 ||u_{xx}(t)||_2^2 \, dt\nonumber\\
		&\leq& L_7 \int_0^{T_0} ||u_x(t)||_2^4 ||u(t)||_{H_\ast^2(\Omega)}^2 \, dt\nonumber\\
		&\leq& L_8 E_u(0)\int_0^{T_0} ||u_x(t)||_2^4 \, dt\nonumber\\
		&\leq& L_9 \int_0^{T_0} ||\Delta u(t)||_2 \,||u(t)||_2 \, dt,\nonumber
	\end{eqnarray*}
	where the last inequality comes from the Gagliardo-Nirenberg inequality and $L_i,~i=7,8,9$ are positive constants. The last inequality yields
	\begin{eqnarray}\label{eq4}
	|I| &\leq& \varepsilon \int_0^{T_0} E_u(t)\,dt + C_{\varepsilon} \int_0^{T_0} ||u(t)||_2^2\,dt\\
	&\leq& \varepsilon T_0 E_u(0)+ C_{\varepsilon} \int_0^{T_0} ||u(t)||_2^2\,dt\nonumber
	\end{eqnarray}
	where $\varepsilon$ is an arbitrary positive constant. Thus, from (\ref{eq3}) and (\ref{eq4}) and making use of the identity of the energy
	\begin{eqnarray*}
		E_u(T_0) - E_u(0) = -\int_0^{T_0} \int_\Omega a(x,y) g(u_t) u_t \,dxdydt,
	\end{eqnarray*}
	we deduce
{\small	\begin{eqnarray*}
		E_u(T_0)(1-\varepsilon T_0)\leq L_{10} \left( \int_0^{T_0}\!\!\! \int_\Omega a(x,y) \left[|u_t|^2 + |g(u_t)|^2\right]\! dxdydt + \int_0^{T_0}\!\!\! \int_\Omega |u|^2 dxdydt \right).
	\end{eqnarray*}}
	
	Choosing $\varepsilon$ sufficiently small and since $E_u(T) \leq E_u(T_0)$ for all $T>T_0$ it follows that
{\small	\begin{eqnarray}
	E_u(T) \leq L_{11}  \left( \int_0^{T} \int_\Omega a(x,y) \left[|u_t|^2 + |g(u_t)|^2\right]\,dxdydt + \int_0^{T} \int_\Omega |u|^2 dxdydt \right),
	\end{eqnarray}}
	for all $T>T_0$.
	
	\underline{Step 3}\;\; It remains to estimate the term $\displaystyle{\int_Q}  \vert u\vert^2\;dx\;dy\;dt$ in terms of the damping term. More precisely, we shall prove the existence of a positive constant $C$ such that
	
	\begin{equation}\label{lem}
	\int_Q  \vert u\vert^2\;dx\;dy\;dt \leq C \int_Q a(x, y) \vert g(u_t)\vert^2\;dx\;dy\;dt,
	\end{equation}
	For this purpose we need the following unique continuation result:
	\begin{lemma}\label{lemma3}
		If the function  $w$  satisfies
{\small		\begin{equation}\label{lem1}
		\left\{
		\begin{array}{lcr}
		w_{tt}(x, y, t) + \Delta^2 w(x, y, t)-p(t) w_{xx}(x, y, t)=0, &\mbox{ in } \;\;Q,& \\\\
		
		w(0, y, t)=w_{xx}(0,y)=w(\pi, y, t)=w_{xx}(\pi, y, t)=0, &(y, t)\in(-l,l)\times (0, T),& \\\\
		
		w_{yy}(x, \pm l, t)+\sigma w_{xx}(x, \pm l, t)=0, &(x, t)\in (0,\pi)\times (0,T),& \\\\
		
		w_{yyy}(x, \pm l, t)+(2-\sigma)w_{xxy}(x, \pm l, t)=0, &(x, t)\in (0,\pi)\times (0, T),&\\\\
		
		w_t(x, y, t)= 0, &\mbox{in}\;\;\omega\times (0, T).&
		\end{array}
		\right.
		\end{equation}}
		Then we have $w=0$ in $Q$.
	\end{lemma}
\noindent{\bf Proof:}
		We follow the arguments of \cite{Tucsnak}. If $p(t)=p_0$ for any $t\in [0, T]$, then the function $v=w_t$ satisfies in the distributions sense the system
{\small		\begin{equation}\label{lem2}
		\left\{
		\begin{array}{lcr}
		v_{tt}(x, y, t)+ \Delta^2 v(x, y, t)-p_0 v_{xx}(x, y, t)=0, &\mbox{ in } \;\;Q,& \\\\
		
		v(0, y, t)=v_{xx}(0, y, t)=v(\pi, y, t)=v_{xx}(\pi, y, t)=0, &(y, t)\in(-l,l)\times (0, T),& \\\\
		
		v_{yy}(x, \pm l, t)+\sigma v_{xx}(x, \pm l, t)=0, &(x, t)\in (0,\pi)\times (0, T),& \\\\
		
		v_{yyy}(x, \pm l, t)+(2-\sigma)v_{xxy}(x, \pm l, t)=0, &(x, t)\in (0,\pi)\times (0, T),&\\\\
		
		v(x, y, t)=0, &\mbox{ in }\;\; \omega\times (0,T).&
		\end{array}
		\right.
		\end{equation}}
		Using Holmgren's uniqueness theorem we conclude  that $v=0$ in $Q$.  From (\ref{lem1}), it follows that
		\begin{equation}\label{lem3}
		\left\{
		\begin{array}{lcr}
		\Delta^2 w(x,y)-p_0 w_{xx}=0, &\mbox{ in } \;\Omega,& \\\\
		
		w(0,y)=w_{xx}(0,y)=w(\pi,y)=w_{xx}(\pi,y)=0, &y\in(-l,l),& \\\\
		
		w_{yy}(x,\pm l)+\sigma w_{xx}(x,\pm l)=0, &x\in (0,\pi),& \\\\
		
		w_{yyy}(x,\pm l)+(2-\sigma)w_{xxy}(x,\pm l)=0, &x\in (0,\pi).&
		
		\end{array}
		\right.
		\end{equation}
		
		The results in \cite{Gazzola1', Gazzola2} show that $w=0$ in $Q$.\\
		
		Let us now suppose that $p^\prime(t)\neq 0$ for $t$ varying in a subset of strictly positive measure of $[0, T]$. The first equation in (\ref{lem1}) and the fact that $w(x, y, t)=w(x, y, 0)$ if $(x, y)\in\omega$ we obtain
		$$\Delta^2 w(x, y, t) -p(t)w_{xx}(x, y, t)=0 \;\;\mbox{in}\;\; \omega\times (0, T).$$
		By deriving with respect to time  the previous equality, we get
		$$p^\prime(t) w_{xx}(x, y, t)=0 \;\;\mbox{in}\;\; \omega\times (0, T).$$
		Taking into account that $p^\prime(t)\neq 0$, we have
		$$w_{xx} (x, y)=0\;\;\;\mbox{in}\;\;\; \omega.$$
		This relation with the boundary conditions in (\ref{lem3}) yields, by Holmgren's uniqueness theorem,
		$$w=0\;\;\;\mbox{in}\;\;\omega.$$
		
		Now, by using Proposition \ref{smooth}, it is possible to find a sequence of sub-domains $(\Omega_{\epsilon_n})_{\epsilon_n>0} $ of $\Omega$  such that  $\Omega\setminus \omega \subset \Omega_{\epsilon_n}$ and  $(\Omega_{\epsilon_n})$ converges to $\Omega$ uniformly, when $\epsilon_n \to 0$. (see figure \ref{fig 5}).\\

		\begin{figure}[h!]
\centering
\includegraphics[width=0.7\linewidth]{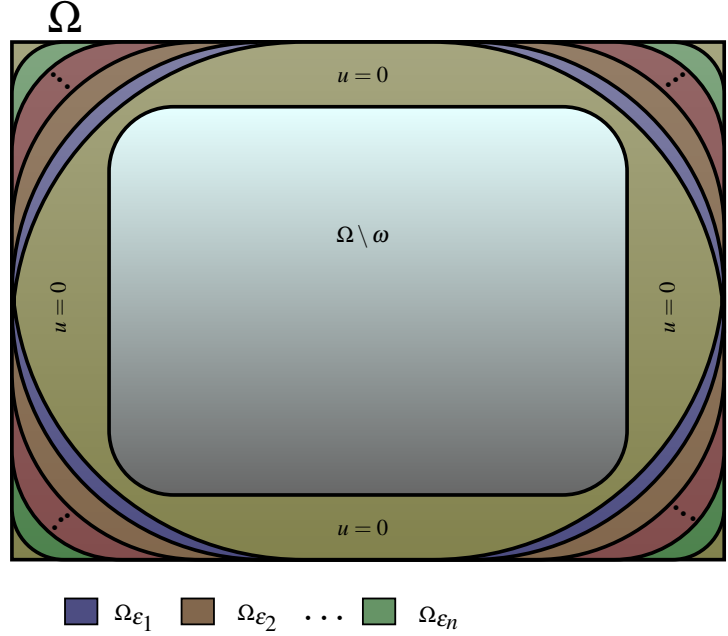}
\caption{Sequence of sub-domains $(\Omega_{\epsilon_n})_{\epsilon_n>0} $ of $\Omega$.}
\label{fig 5}
\end{figure}

		Furthermore, since $w=0$ in $\omega$, we have $w=\partial_\nu w =0$ on $\partial \Omega_\epsilon$ for all $\epsilon>0$.\\
		
		Now using Theorem 2.1 of \cite{Kim}, we obtain that $w=0$ in $\Omega_\epsilon$, for all $\epsilon>0$. Hence, by the uniform convergence, we have $w=0$ in $Q$. $ \qedsymbol. $
		
		\medskip

	Now let us suppose that (\ref{lem}) is not satisfied and let $(u_0^k, u_1^k)_k $ be a sequence of initial data where the corresponding solutions $(u^k)_k$ of (\ref{main problem}), with $E_u^k(0)$ assumed uniformly bounded in $k$, satisfy
	\begin{equation}\label{lem4}
	\lim_{k\to+\infty} \frac{\displaystyle{\int_Q}  \vert u^k\vert^2 \;dx\;dy\;dt}{\displaystyle{\int_Q} a(x, y) \vert g(u^k_t)\vert^2\;dx\;dy\;dt }=+\infty.
	\end{equation}
	
	Define $$\lambda_k=\left[\displaystyle{\int_Q}  \vert u^k\vert^2 \;dx\;dy\;dt\right]^{\frac{1}{2}},$$
	and
	$$w^k=\frac{u^k}{\lambda_k}.$$
	By using the above equalities and (\ref{lem4}), we have
	\begin{equation}\label{lem5}
	\displaystyle{\int_Q}  \vert w^k\vert^2\;dx\;dy\;dt=1\,
	\end{equation}
	and
	\begin{equation}\label{lem6}
	\displaystyle{\int_Q}  a(x, y) \frac{\left\vert g(u^k_t)\right\vert^2}{\lambda_k^2}\;dx\;dy\;dt\to 0\,\;\;\;\mbox{as}\;\; k\to +\infty.
	\end{equation}
	Besides $w^k$ satisfies
{\small	\begin{equation}\label{lem7}
	\left\{
	\begin{array}{lcr}
	w^k_{tt}(x, y, t) + \Delta^2 w^k(x, y, t)+ \left(-P + S\lambda_k^2 \displaystyle{\int_\Omega} (w^k_x)^2\;dx\;dy\right) w^k_{xx}(x, y, t)\\+ a(x, y) \frac{g(u_t^k)}{\lambda_k}=0,\; \qquad\qquad\qquad\qquad\qquad\qquad\qquad\qquad\qquad\,\,\,\,\,\qquad\qquad\mbox{ in } \;Q, \\\\
	
	w^k(0, y, t)=w^k_{xx}(0,y)=w^k(\pi, y, t)=w^k_{xx}(\pi, y, t)=0, \;(y, t)\in(-l,l)\times (0, T), \\\\
	
	w^k_{yy}(x, \pm l, t)+\sigma w^k_{xx}(x, \pm l, t)=0,\; \qquad\qquad\qquad\qquad\,\quad(x, t)\in (0,\pi)\times (0,T), \\\\
	
	w^k_{yyy}(x, \pm l, t)+(2-\sigma)w^k_{xxy}(x, \pm l, t)=0, \qquad\qquad\qquad\!\;(x, t)\in (0,\pi)\times (0, T).
	
	\end{array}
	\right.
	\end{equation}}

	As in the previous section, we have similar convergence results for the sequence $(w^k)_k$ as in (\ref{20}), (\ref{21}) and (\ref{22}). We denote by $w$ the limit of $(w^k)_k$. In addition since $(\lambda_k$) is bounded in $\mathbb{R}$, we obtain, by extracting a subsequence still denoted by $(\lambda_k)_k$, that
	$$\lambda_k \to \lambda\;\;\mbox{in}\;\;\mathbb{R}, \;\;\mbox{when}\;\;k\to +\infty.$$
	Passing to the limit in the equation, when $k\to +\infty$, we get
{\small	\begin{eqnarray*}
\left\{
		\begin{array}{lcr}
			w_{tt}(x,y,t)+\Delta^2 w(x,y,t)-p(t)w_{xx}(x, y, t)=0, &\mbox{ in } \;\Omega \times (0, T),& \\\\
			
			w(0,y,t)=w_{xx}(0,y,t)=w(\pi,y,t)=w_{xx}(\pi,y,t)=0, &(y,t)\in(-l,l)\times (0, T),& \\\\
			
			w_{yy}(x,\pm l,t)+\sigma w_{xx}(x,\pm l, t)=0, &(x,t)\in (0,\pi)\times (0, T),& \\\\
			
			w_{yyy}(x,\pm l,t)+(2-\sigma)w_{xxy}(x,\pm l, t)=0, &(x,t)\in(0,\pi)\times(0, T),&\\\\
			
			w_t(x,y,t)=0, &\mbox{ in } \;\omega\times (0, T),&
			
		\end{array}
		\right.
	\end{eqnarray*}}
	where $p(t)= P -S\lambda^2 \displaystyle{\int_\Omega} w_x^2\;dx\;dy$.\\
	
	Using Lemma \ref{lemma3}, we have $w=0$ in $Q$, which is in contradiction with (\ref{lem5}) and the fact that $ (w^k)_k$ converges strongly to $w$ in $L^2(0, T; L^2(\Omega))$ and consequently (\ref{main ineq}) holds true. ~$ \qedsymbol  $
\\
	

{\color{black}
Henceforth we will also use the notation
\begin{equation}\label{def:damping}
\bfD_{a}^{b}\big(g(s); u_{t}\big) \dfn
\int_{a}^{b}\int_{\Omega}a(x,y)g(u_{t}) u_{t} \,dx\,dt,
\end{equation}
and the identity of the energy (\ref{indentity of energy}) now reads as follows:
\begin{equation}\label{energy-identity}
E(t_{2}) + \bfD_{t_1}^{t_2}\big(g(s);u_{t}\big)= E(t_{1}) \quad \text{ for all } t_2\geq t_1\geq
0,
\end{equation}

\medskip

The main result of this paper explicitly quantifies the asymptotic decay rates of the finite energy for the system \eqref{main problem}.
\begin{thm}\label{thm:decay}
Denote by $(u,u_t)$ a weak solution of the problem \eqref{main problem}.  Suppose the $a = a(x, y) \in  L^\infty(
\Omega)$ is assumed to be a nonnegative
bounded function such that $a(x,y) \geq a_0 > 0$ a.e. in $\omega$ for some non empty open subset $\omega$ around the boundary $\partial \Omega$� of $\Omega$
 and some positive constant $a_0 > 0$.   Define $h$  to be concave, strictly increasing function,  vanishing at $0$ and such that
\begin{equation}\label{def:h0}
h(sg(s)) \geq s^{2} + g(s)^2,\text{ for }\left| s\right| \leq 1,
\end{equation}
(which can always be constructed since $g$ is continuous increasing $g(0)=0$).

In addition, if $g$  is not linearly bounded at infinity (of order  $\cO$ other than $1$ according to the Definition \ref{def:order}), then let the Assumption \ref{as:regularity} be satisfied with the corresponding integrability indices $p_{0}$.  Next, define

\[
\bfC = \|u_t\|_{L^\infty\big(\bbR_+; L^{p_{0}}(\Omega) \big)}^{\frac{|1-\cO(g)|}{p_{0}-1-\cO(g)}},
\qquad \tl{h}(s)
= s^{\frac{p_{0}- 2 \max\{\cO(g),1\} }{p_{0}-1-\cO(g)}}.
\]

Conclusion: then there exist constants $T_0 \geq T >0$ such that the energy $E(t)$ given by \eqref{energy} satisfies
\[
E_u(t)\leq S\left( \frac{t}{T}-1\right),\quad \forall t>T_0,
\]
where $\lim_{t\to \infty}S(t)=0$. Moreover, suppose for some $\mathfrak{f}\in \{Id, h,\tl{h}\}$
\[
\lim_{s\to 0^{+}} \frac{[s+ h(s) + \tl{h}(s)] - \mathfrak{f}(s)}{\mathfrak{f}(s)}=0,
\]
then $S(t)$ solves the monotone ODE
\begin{equation}\label{ode}
\frac{d}{dt}S(t)+H^{-1}\big((1-\del)S(t)\big)=0,\quad S(0)=E_u(0).
\end{equation}
where parameter $\del>0$ can be chosen to be arbitrarily small at the expense of growing $T_{0}$. The map $H$, in this case is given by
\[
 H(s) = C(\bfC)C_{L}^{2} \mathfrak{f}(s)
\]
where $C>0$ depends only on the functions $I,h,\tl{h}$), while $C_{L}$ is the linear observability constant from (\ref{3}). (Essentially $H$ is proportional to the map whose growth near the origin is the fastest from among $I, h,\tl{h}$).
\end{thm}
\begin{proof} {Initially, before to prove this theorem, let's give some examples in order to clarify our ideas.}
\subsection{Examples of energy decay rates}\label{sec:decay-examples}

\medskip

\subsubsection{Linearly bounded damping}

\medskip

If the feedback is linear (or bounded above and below by linear maps with positive slopes), e. g., $g(s)=s,$ then the function $H$ in \eqref{ode} is linear, hence $S$ solves an equation of the form $S' + CS = 0$ which has an exponentially decaying solution. Specifically, there exists a constant $C= C(E_u(0))$ dependent on the initial energy and some $k>0$ such that%
\[
E(t)\leq Ce^{-kt}E(0) \qquad t>0
\]
In this setting no assumptions on the regularity of solutions, beyond the finite energy level are necessary.

\subsubsection{Nonlinear damping near the origin}

\medskip

The decay rates computed in the Table \eqref{table1} assume that the feedback map is linearly bounded at infinity, i.e. $\cO(g)=1$  or, equivalently,  $a|s| \leq g(s) \leq b|s|$ for $|s|>1$, with some positive constants $a,b$.

\begin{table}[h]
\begin{center}
\begin{tabular}{|c|c|c|c|}\hline
&	\multicolumn{3}{c|}{feedback map is linearly bounded at infinity (for $|s|>1$)}\\
&	\multicolumn{3}{c|}{feedback near the origin is not linearly bounded (for $|s|\leq1$)}\\ \hline
&	 &  &      \\
&	$g(s)$\,  $=s^{\theta<1}$   &   $g(s)$\,  $=s^{r>1}$                    &    $g(s)$   $=s^3 e^{-1/s^2}$ \\
  \hline regularity &  \multicolumn{3}{c|}{finite-energy}  \\ \hline &&&\\
 $h(s)$              &	  $2s^{\frac{2\theta}{\theta+1}}$                                  &   $2s^{\frac{2}{r+1}}$           &                                                             \\ \hline &&&\\
$H^{-1}((1-\del)s)$              &	  $c s^{\frac{\theta+1}{2\theta}}$           &   $c s^{\frac{r+1}{2}}$         &       $c_1 s^2 \exp(-c_2/s)$                                                       \\ \hline &&&\\
$S(t)$  in \eqref{ode} &  $\ds \bigg[ {\ss\frac{c(1-\theta)}{2\theta }}(t+c_0)\bigg]^{-\frac{2\theta}{1-\theta}}$ & $\ds \bigg[ {\ss\frac{c(r-1)}{2 }}(t+c_0)\bigg]^{-\frac{2}{r-1}}$
&  $\ds \frac{c_2}{\ln(c_1 c_2 t + c_0)}$\\  &&&\\\hline
\end{tabular}
\end{center}
\caption{Asymptotic energy decay rates in the case when the feedback $g(s)$  linearly bounded at infinity (for $|s|>1$) and   is not linearly bounded only near the origin (for $|s|\leq1$).}
\label{table1}
\end{table}

\subsubsection{Sublinear or superlinear damping at infinity}

The asymptotic decay rates computed in Table \ref{table2}  assume that the feedback maps is linearly bounded at the origin, and  has the order other than $1$ at infinity according to the definition \eqref{def:order}. In this case uniform decay in finite-energy space requires uniform regularity of solutions in stronger topology.

\begin{table}[h]
\begin{center}
\begin{tabular}{|c|c|c|}\hline
&	\multicolumn{2}{c|}{feedback linearly bounded near the origin (for $|s| <1$),} \\
&	\multicolumn{2}{c|}{feedback is not linearly bounded at infinity (for $|s|\geq 1$)} \\
 \hline &&\\
                                 &	 $g(s)$\,  $=s^{\theta<1}$   &    $g(s)$\,  $=s^{r>1}$                     \\ \hline&&\\
                                 regularity &   $u_t \in L^\infty (\mathbb{R}_+; L^{p_0}(\Omega))$ & $u_t \in L^\infty (\mathbb{R}_+;  L^{p_0}(\Omega))$  \\  &&\\
                                  &  $q\dfn p_0$ \text{ or }$ p > 2$  & $q\dfn p_0$ \text{ or }$ p > 2r$  \\ \hline&&\\
 $\tl{h}(s)$              &	  $s^{\frac{q-2}{q-\theta-1}}$                                  &   $s^{\frac{q-2r}{q-r-1}}$                       \\ \hline &&\\
$H^{-1}((1-\del)s)$              &	  $c s^{\frac{q-\theta-1}{q-2}}$           &   $c s^{\frac{q-r-1}{q-2r}}$               \\ \hline  &&\\
$S(t)$  in \eqref{ode} &  $\ds \bigg[ {\ss\frac{c(1-\theta)}{q-2}}(t+c_0)\bigg]^{-\frac{q-2}{1-\theta}}$ & $\ds \bigg[ {\ss\frac{c(r-1)}{q-2r }}(t+c_0)\bigg]^{-\frac{q-2r}{r-1}}$
\\  &&\\\hline &\multicolumn{2}{c|}{}\\
Strong  data &
\multicolumn{2}{c|}{$u_t\in H_*^2(\Omega^{dim =2})\into L^{q<\infty}(\Omega) $}\\
&\multicolumn{2}{c|}{Arbitrarily fast algebraic rate}\\
&\multicolumn{2}{c|}{ (but Sobolev constant blows up as $q\nearrow \infty$)}\\
\hline
\end{tabular}
\end{center}
\caption{Asymptotic energy decay rates in the case feedback map  $g(s)$ is linearly bounded at the origin (for $|s| <1$)  and  is not linearly bounded, and only at infinity  (for $|s|\geq 1$).}
\label{table2}
\end{table}

\subsubsection{Combining different types of damping}

As a consequence of the Theorem \ref{thm:decay}, when different types of nonlinearities at the origin and at infinity are present, and possibly different for the feedback $g$, the overall decay rate can be guaranteed to be the slowest one of the individual rates computed individually for each nonlinearity in the Tables \ref{table1} and \eqref{table2}.

\section{Proof of uniform energy decay}\label{sec:decay}
\medskip
\subsection{Bridging linear and nonlinear observability inequalities}
\medskip
The stability result for the energy of \emph{nonlinear} system follows from a stabilization estimate for a \emph{linear} system as we proved in section 2, namely:
\begin{lemma}[Linear observability estimate]\label{lem:linear}
Assuming that $g(s)=s$, there exists a sufficiently large $T>0$, and a constant $C_L$ dependent on $T,L$ such that the energy of the solution to \eqref{1} satisfies
\[
E(T)\leq C_L \dampinglin
\]
\end{lemma}
The proof of the \emph{linear result} have been addressed in Section 2. The goal of this section is to verify the following extension to the non-linear case.

\begin{lemma}[Nonlinear observability]\label{lem:nonlin}
If the map $g$ is not linearly bounded at infinity (of order other than $1$ according to the Definition \ref{def:order}), then let the Assumption \ref{as:regularity} be satisfied with the corresponding integrability indices $p_{0}$.  Let $T$ and $C_L$ be given by Lemma \eqref{lem:linear}. Then for some constant $C>0$ the solution to \eqref{main problem} satisfies
\[
\begin{split}
E_u(0)\leq C L_T^{2} &\Bigg[ (h+I)\big\{ \damping\big\}\\
 &+ (\sgn_{\infty}[g] )\|u_{t}\|_{L^{\infty}(\bbR_{+}; L^{p_{0}}(\Omega))}^{\frac{|\cO(g)-1|}{p_0-1-\cO(g)}}
 \bigg(\damping\bigg)^{\frac{p_{0}-2\max\{1,\cO(g)\}}{p_{0}-1-\cO(g)} }\Bigg],
 \end{split}
\]
where   $\sgn_\infty(G) \equiv  0$ if $G$ is linearly bounded at infinity, i.e. $\cO(G) = 1$, and $\sgn_{\infty}(G) \equiv 1$ otherwise.
\end{lemma}

In order to prove Lemma \ref{lem:nonlin} we shall exploit the nonlinear observability inequality given in (\ref{main ineq}).  To justify the above aforementioned inequality and the proof of the lemma it is necessary:
\begin{itemize}
	\item  to have an energy identity for \emph{weak} solutions of the original nonlinear system \eqref{main problem},
	
\end{itemize}
When the damping term is linearly bounded the condition follows from the regularity  furnished by the well-posedness to problem (\ref{main problem}) previously established. When the nonlinearity is stronger,  the regularity Assumption  \ref{as:regularity} comes into play;  as a consequence  $g(u_t)$ belongs to $L^{1}(\bbR_+;L^{2}(\Omega))$. With this extra regularity one can extend the  energy  identity \eqref{energy-identity}  to weak solutions by employing finite-difference approximations, exactly as in \cite{Bociu}.  In fact, just for the purposes of the weak energy \emph{inequality} the argument simplifies  if, for instance, the map $s\mapsto g(s)s$ is convex since then one can appeal to weak lower-semicontinuity of the associated functionals without invoking regularity.

To conclude the proof of Lemma \eqref{lem:nonlin}  split
\[
X:= \Omega \times ]0,T[, \quad X = X_0 \cup X_\infty
\]
where (for any a.e. defined version of $u_t$)
\[
X_0 \dfn \big\{ (x,t) \in  X \quad :  \quad  |u_t(x,t)| < 1 \big\}
\]
and $X_\infty \dfn X\; \setminus X_0$. The proof will require the following inequalities:

\begin{enumerate}[I.]
	\item  \textbf{Damping near the origin}. By construction of the concave function $h$ (\eqref{def:h0} we have
\begin{equation}\label{i:origin}
\begin{split}
\int_{X_0} a(x,y)(g(u_t)^2 + u_t^2 ) dX \leq&    \int_{X_{0}} h(g(u_t)u_t) a(x,y)\dX \\
\leq &  C_{a,T} h\left(\int_{X} a(x,y)(x)g(u_t)u_t\dX\right),
\end{split}
\end{equation}
where the last step invoked Jensen's inequality, and $C_{a,T} =  \int_{X_{0}} a \dX$.

\item \textbf{Linearly-bounded damping at infinity}. If $\cO(g) =1$ according to the Definition \eqref{def:order}, then  $g(s)^2 + s^2 \leq c g(s)s$ for some constant $c>0$ provided $|s|>1$.  Directly estimate:
\begin{equation}\label{i:lin}
\begin{split}
\int_{X_\infty}a(x,y) (g(u_t)^2 + u_t^2)   \dX
\leq& c\int_{X} a(x,y)g(u_t)u_t \dX.
\end{split}
\end{equation}

\item \textbf{Superlinear damping at infinity}.
Suppose  $\cO(g) = r> 1$ according to the Definition \eqref{def:order}. Then  $g(s) > c s$ for $|s|>1$, some $c>0$ independent of $s$, and we trivially estimate
\begin{equation}\label{i:superlin:1}
\begin{split}
\int_{X_\infty}a(x,y) u_t^2 \dX
\leq&
c'\intX a(x,y)g(u_t)u_t \dX.
\end{split}
\end{equation}
Next, for any $\lam \in ]0,1[$
\begin{equation}\label{i:superlin:2}
 \int_{X_\infty} a(x,y)g(u_t)^2 \dX  =\overbrace{\int_{X_\infty}  a(x,y)|g(u_t)|^{2\lam} |g(u_t)|^{2(1-\lam)}\dX}^{J_1}.
\end{equation}
Choose any $p>2r$ and estimate the integral labeled $J_1$ using H\"older's inequality with conjugate exponents $\frac{p}{2\lam r}$ and $\frac{p}{p-2\lam r}$ (splitting $a$ as $a^{2\lam r/p} \cdot a^{(p-2\lam r)/p}$):
\begin{equation}\label{i:superlin:3}
J_1 \leq \left(\int_{X_\infty} a(x,y)|g(u_t)|^{p/r} \dX\right)^{2\lam r /p}\left(\int_{X_\infty}
a(x,y)|g(u_t)|^{\frac{2 (1-\lam)p}{p-2\lam r}} \dX\right)^{\frac{p-2\lam r}{p}}.
\end{equation}
Note that $\cO(g)=r$ implies
\begin{equation}\label{g-sim-superlin}
g(s)s \sim s^{r+1} \sim g(s)^{(r+1)/r},\quad  |s|>1.
\end{equation}
Thus,  for  $|g(u_t)|^{\frac{2 (1-\lam)p}{p-2\lam r}}$ to be equivalent to the dissipation integrand $g(s)s$ we solve
\[
 \frac{2(1-\lam)p}{p-2\lam r } = \frac{1+ r}{r}    \implies  \lam = \frac{p(r-1)}{2r(p-r-1)}.
\]
With this choice of $\lam$ combine \eqref{i:superlin:1}, \eqref{i:superlin:2} and \eqref{i:superlin:3} to  conclude
\begin{equation}\label{i:superlin:final}
\begin{split}
&\int_{X_\infty} a(x,y)(g(u_t)^2 + u_t^2) \dX \\
\leq &c\cdot C_{a} \|u_t\|_{ L^\infty(\bbR_+;L^p)}^{ \frac{\cO(g)-1}{p-1-\cO(g)}} \left(\intX a(x,y)  g(u_t)u_t \dX \right)^{\frac{p-2\cO(g)}{p-1-\cO(g)}}.
\end{split}
\end{equation}
for some constant $c$ (dependent only on  \eqref{g-sim-superlin}) and $C_{a} = (\sup a)^{2\lam r/p}$. The resulting inequality holds provided the $L^\infty(\bbR_+; L^p(\Omega))$-norm, of $u_t$ is finite for $p>2 \cO(g)$.
	
\item \textbf{Sublinear damping at infinity}.
Assume  $\cO(g) = r < 1$ according to the Definition \eqref{def:order}. Then  $c|s| >  |g(s)|$ for $|s|>1$, some $c>0$ independent of $s$:
\begin{equation}\label{i:sublin:1}
\begin{split}
\int_{X_\infty} a(x,y)g(u_t)^2 \dX
\leq&c \intX a(x,y)g(u_t)u_t \dX.
\end{split}
\end{equation}
For any  $\lam \in ]0,1[$
\begin{equation}\label{i:sublin:2}
 \int_{X_\infty} a(x,y)u_t^2 \dX = \overbrace{\int_{X_\infty} a(x,y) |u_t|^{2\lam} |u_t|^{2(1-\lam)}\dX}^{J_2}.
\end{equation}
Let $p>2$ and estimate the integral labeled $J_2$ using H\"older's inequality with exponents   $\frac{p}{2\lam}$, and $\frac{p}{p-2\lam}$:
\begin{equation}\label{i:sublin:3}
J_2 \leq \left(\int_{X_\infty} a(x,y)|u_t|^p \dX\right)^{2\lam/p}\left(\int_{X_\infty}a(x,y)
|u_t|^{\frac{2 (1-\lam)p}{p-2\lam}} \dX\right)^{\frac{p-2\lam}{p}}.
\end{equation}
The value of $\lambda\in]0,1[$ is chosen to ensure that
\begin{equation}\label{g-sim-sublin}
|u_t|^{\frac{2 (1-\lam)p}{p-2\lam}}  = u_t^{r+1} \sim g(u_t)u_t,\quad\text{for}\quad |u_t|>1
\end{equation}
namely
\[
 \frac{2(1-\lam)p}{p-2\lam} = 1+ r    \implies  \lam = \frac{p(1-r)}{2(p-1-r)}.
\]
Combine \eqref{i:sublin:1}, \eqref{i:sublin:2}, \eqref{i:sublin:3}
\begin{equation}\label{i:sublin:final}
\begin{split}
 &\int_{X_\infty} a(x,y)(g(u_t)^2 + u_t^2)\dX \\
 \leq &  c \cdot C_a
  \|u_t\|_{L^\infty(\bbR_+; L^p)}^{\frac{1-\cO(g)}{p-1-\cO(g)}}\left(\intX
a(x,y)g(u_t)u_t \dX\right)^{\frac{p-2}{p-1-\cO(g)}}.
   \end{split}
\end{equation}
for some $c>0$ (dependent on the estimate \eqref{g-sim-sublin}),  $C_{a} = (\sup a)^{2\lam/p}$, and asserting that $\|u_t\|_{L^\infty(\bbR_+; L^p)}<\infty$ with  $p>2$.
\end{enumerate}
\bigskip

Having established the above estimates, return to energy inequality \eqref{main ineq}, combine it with the identity of the energy (\ref{energy-identity})  the inequality \eqref{i:origin}, and with either \eqref{i:lin}, or \eqref{i:superlin:final}, or \eqref{i:sublin:final}, depending on whether $\cO(g)=1$, $\cO(g)>1$, or $\cO(g)<1$  respectively.  Using the definition  \eqref{def:damping}, and after relabeling of constants
\[
\begin{split}
E_u(0)\leq L_T \bigg[ & C_{a,T} (h+I)\{ \damping  \}\\
 &+ C_{a} \|u_{t}\|_{L^{\infty}(\bbR_{+}; L^{p_{0}}(\Omega))}^{\frac{|\cO(g)-1|}{p_0-1-\cO(g)}}
 \bigg(\damping\bigg)^{\frac{p_{0}-2\max\{1,\cO(g)\}}{p_{0}-1-\cO(g)} }\bigg].
\end{split}
\]

Thus, the conclusion of Lemma \ref{lem:nonlin} yields. \qed

\subsection{Deriving the energy decay rates}
\medskip
The result of Lemma \ref{lem:nonlin} can be recast into the form
\[
 E_u(0)\leq F_T\left(\damping\right)
 \stackrel{\eqref{energy-identity}}{=}  F_T\left( E_u(0) - E_u(T)\right)
\]
with
\[
F_T \dfn  CC_L^2( I  + h  + \mathbf{C} \tl{h})
\]
\begin{equation}\label{def:tl-h0}
\bfC = \|u_t\|_{L^\infty\big(\bbR_+; L^{p_{0}}(\Omega) \big)}^{\frac{|1-\cO(g)|}{p_{0}-1-\cO(g)}},
\qquad \tl{h}(s)
= s^{\frac{p_{0}- 2 \max\{\cO(g),1\} }{p_{0}-1-\cO(g)}}.
\end{equation}
The function $F_T$ is monotone increasing, zero at the origin. Due to the energy being non-increasing we have, a fortiori,
\[
E(T) \leq F_T(E(0) - E(T))\quad \text{or}\quad
   (I  + F_T^{-1}) E(T) \leq E(0).
\]
Henceforth $E(t)$ will denote the energy of the original nonlinear ``$u$"-problem \eqref{main problem}. Now we may appeal to the result of \cite{Lasiecka-Tataru}
 to conclude that the energy $E(t)$ is decaying to $0$, as $t\to \infty$, at least as fast as a solution to a certain nonlinear ODE. Rather than stating the ODE in the full form which typically does not admit closed-form solutions, let us restate an approximate version (\cite{Lasiecka-Tataru}):
\[
 E(t) \leq S\left(\frac{t}{T} - 1\right),\quad t \geq T_0 > T
\]
for a sufficiently large $T_0$  and a function $S$ that solves the (monotone) non-linear ODE
\[
S_t + H^{-1}( (1-\del)S) = 0,\qquad  S(0) = E_u(0).
\]
Here the parameter $\del>0$ can be made arbitrarily small at the expense of a growing $T_0$, and the function $H$ has the fastest growth near the origin from among $I$, $h$, $h$ and $\tl{h}$. For the case when $H=\tl{h}$  we can solve the ODE explicitly using the Definition \ref{def:tl-h0}. Essentially, the resulting rate will be the slowest from among exponential (if we take $H\sim I$),  and those guaranteed by $H\sim h$, or $H\sim \tl{h}$ . This observation concludes the proof of Theorem \ref{thm:decay}. \end{proof} 

\section{Numerical Results}

\subsection{Description of the numerical scheme.}

In this section, we will replicate numerically the results obtained in the previous sections. In particular, and given the boundary conditions we have to deal with, our proposal consist on the approximation of the solution of Problem \eqref{main problem} using the finite differences method. To achieve this, the $x$ domain $[0,\pi]$ will be subdivided in $J+1$ equally spaced sub-intervals with length $\Delta x$ each, while the $y$ domain $[-l,l]$ will be subdivided in $K+1$ sub-intervals, each of length $\Delta y$.

The domain $\Omega$ will be then discretized using rectangles of area $\Delta x \Delta y$. We will also write $x_j := j\Delta x,\; j = 0,1,\dots,J+1$ and $y_k := -l + k\Delta y,\; k = 0,1,\dots,K+1$.

Integrating from $t=0$ to some $t=T \in \RR^+$ using $N$ timesteps of length $\Delta t:= \frac{T}{N}$, the solution at a timestep $n$ will be approximated by a vector
\begin{equation*}
 U^n \in \RR^{(J+2)(K+2)}: U^n = [U^n_0\; U^n_1 \; \dots \; U^n_{(K+1)}]^T
\end{equation*}
where each $U_k^n$ is such that
\begin{equation}\label{vecU}
  U^n_k \in \RR^{(J+2)}: \; U_k^n = [U^n_{0,k}\; U^n_{1,k} \; \dots \; U^n_{J+1,k}]
\end{equation}
this is, each $U_k^n$ describes, for each node $k$ on the $y$ coordinate, the solution for all of the nodes on the $x$ coordinate.

\subsubsection{Discretization of the bilaplacian.}

Recalling that $\Delta^2 u = u_{xxxx} + 2u_{xxyy} + u_{yyyy}$, we will proceed to discretize directly each term using centered finite differences. Given a function $f(x)$ defined over $[0,\pi]$, we will write $f_i := f(x_i),\: x_i \in (0,\pi), \: i = 0,\:2,\: \dots,\: J+1$. Ignoring the boundary for now, its fourth derivative at the $j$-th node can be approximated as follows
\begin{equation*}
  f_{xxxx}(x_i) \approx \frac{f_{i-2} - 4f_{i-1} + 6f_i - 4f_{i+1} + f_{i+1}}{\Delta x^4}, \qquad i = 0,1,\dots,J+1
\end{equation*}
this can be also represented as a matrix-vector product:
\begin{equation*}
  f_{xxxx} \approx  \frac{1}{\Delta x^4}D^4 f := \frac{1}{\Delta x^4}
  \begin{bmatrix}
    6 & -4 & 1 &    & & \\
    -4 & 6 & 4 & 1  & & \\
    1 & -4 & 6 & -4 & 1 & \\
    & \ddots & \ddots & \ddots & \ddots & \ddots \\
    & & 1 & -4 & 6 & -4 & 1  \\
    & & & 1 & -4 & 6 & -4 \\
    & & & & 1 & -4 & 6
  \end{bmatrix}
  \begin{bmatrix}
    f_1 \\ f_2 \\ f_3 \\ \vdots \\ f_{J-1} \\ f_{J} \\ f_{J+1}
  \end{bmatrix}
\end{equation*}
The second derivative receives also the same treatment: for a centered scheme, we have
\begin{equation}\label{der2}
  f_{xx}(x_i) \approx \frac{f_{i-1} - 2f_i + f_{i+1}}{\Delta x^2}
\end{equation}
and as a matrix-vector product, we have
\begin{equation}\label{d2}
  f_{xx} \approx \frac{1}{\Delta x^2}D^2 f := \frac{1}{\Delta x^2}
  \begin{bmatrix}
    -2 & 1 & & & \\
    1 & -2 & 1 & & \\
    & \ddots & \ddots & \ddots \\
    &  & 1 & -2 & 1 \\
    & & & 1 & -2
  \end{bmatrix}
  \begin{bmatrix}
    f_1 \\ f_2 \\ \vdots \\ f_{J} \\ f_{J+1}
  \end{bmatrix}
\end{equation}
This can be extended to further dimensions, while analog definitions can be given for $f_{yy}$ and $f_{yyyy}$. With this in consideration, and given the structure of the numerical solution $U$, its bilaplacian can be approximated as a pentadiagonal block matrix:
\begin{align}\label{bilapD}
  \Delta^2 U &= D_x^4 U + D_y^4U + 2D_x^2D_y^2 U 
\end{align}
where, for the identity matrix $I\in \RR^{(J+2)(K+2) \times (J+2)(K+2)}$,
\begin{align*}
  D_x^4 U = \frac{1}{\Delta x^4}I \otimes D^4 , \quad &D_y^4 U = \frac{1}{\Delta y^4} D^4 \otimes I \\
    D_x^2 U = \frac{1}{\Delta x^2}I \otimes D^2 , \quad &D_y^2 U = \frac{1}{\Delta y^2} D^2 \otimes I
\end{align*}

\subsection{Treatment of the boundary}

Given the boundary conditions of Problem \eqref{main problem}, we must proceed to modify the discretized bilaplacian. On the $x$ coordinate, we know that $u(0,y,t) = u(\pi,y,t) = 0$. Hence, we get $U_{0,k}^n = U_{J+1,k}^n = 0,\; \forall k \in [0,K+1],\; \forall n \in [0,N]$. This doesn't alter the form of the matrix representing the second derivative if we apply it for $U_{i,k}^n,\; i \in [1,J]$, but this also forces us to do the same for the fourth derivative matrix. From here, we will denote $[U^n]_{j,k}$ as the $j$-th element of the vector $U_k^n$ defined in \eqref{vecU}. Regarding that case, for $i=1$ and $i=J$ we have:
\begin{align*}
  [D_x^4 U^n]_{1,k} &= \frac{U^n_{-1,k} - 4U^n_{0,k} + 6U^n_{1,k} - 4U^n_{2,k} + U^n_{3,k}}{\Delta x^4} \\
  [D_x^4 U^n]_{J,k} &= \frac{U^n_{J-2,k} - 4U^n_{J-1,k} + 6U^n_{J,k} - 4U^n_{J+1,k} + U^n_{J+2,k}}{\Delta x^4}.
\end{align*}
In order to get the values of $U^n_{-1,k}$ and $U^n_{J+2,k}$, we have to take a look at the discretized second derivative on the boundary. Because $u_{xx}(0,y,t) = u_{xx}(\pi,y,t) = 0$, we can write
\small
\begin{equation*}
  [D_x^2U^n]_{0,k} = \frac{U^n_{-1,k} - 2U^n_{0,k} + U^n_{1,k}}{\Delta x^2} = 0, \quad   [D_x^2U^n]_{J+1,k} = \frac{U^n_{J,k} - 2U^n_{J+1,k} + U^n_{J+2,k}}{\Delta x^2} = 0
\end{equation*}
\normalsize
and thus, $U^n_{-1,k} = -U^n_{1,k} $ and $U_{J+2,k}^n = U_{J,k}^n$. Hence, the matrix representation will be given with the aid of a matrix $\hat{D}^4 \in \RR^{J\times J}$ such that
\begin{equation}\label{d4x}
  D_x^4 = \frac{1}{\Delta x^4} I \otimes
  \begin{bmatrix}
    5 & -4 & 1 &    & & \\
    -4 & 6 & 4 & 1  & & \\
    1 & -4 & 6 & -4 & 1 & \\
    & \ddots & \ddots & \ddots & \ddots & \ddots \\
    & & 1 & -4 & 6 & -4 & 1  \\
    & & & 1 & -4 & 6 & -4 \\
    & & & & 1 & -4 & 5
  \end{bmatrix}
  =: \frac{1}{\Delta x^4} I \otimes \hat{D}^4
\end{equation}
On the $y$ coordinate, the second derivative can be modified with ease when considering he boundary condition $u_{yy}(x,\pm l,t) + \sigma u_{xx}(x,\pm l,t) = 0$. With this, we have for $j\in [1,J]$ and for $n \in [0,N]$ that
\begin{equation*}
[D_y^2 U^n]_{j,0} = -\sigma [D_x^2 U^n]_{j,0},\quad [D_y^2 U^{n}]_{j,K+1} = -\sigma [D_x^2 U^{n}]_{j,K+1}
\end{equation*}
and thus, the matrix representation of the second derivative over $y$ will be
\begin{equation}\label{d2y}
  D_y^2= \frac{1}{\Delta y^2}
    \begin{bmatrix}
    -\frac{\sigma}{\Delta x^2} D^2 &  & & & \\
    I & -2I & I & & \\
    & \ddots & \ddots & \ddots \\
    &  & I & -2I & I \\
    & & &  & -\frac{\sigma}{\Delta x^2} D^2
    \end{bmatrix}
\end{equation}
where the matrix $D^2$ was already presented in \eqref{d2}. For the fourth derivative, we will have the same problem as in the $x$ coordinate case; this is,
\begin{align*}
  [D_y^4 U^n]_{j,0} &= \frac{U_{j,-2}^n - 4U_{j,-1}^n + 6U_{j,0}^n - 4U_{j,1}^n + U_{j,2}^n}{\Delta y^4} \\
  [D_y^4 U^n]_{j,K+1} &= \frac{U^{n}_{j,K-1} - 4U^n_{j,K} + 6U^{n}_{j,K+1} - 4U^{n}_{j,K+2} + U^{n}_{j,K+3}}{\Delta y^4}
\end{align*}
To compute $U_{j,k}^{n}$ when $k = -2,-1,K+2,K+3$, we need to combine the fourth derivative discretization at the boundary with the one obtained from the second derivative discretization. This gives the following matrix representation
\tiny
\begin{equation}\label{d4y}
  D_y^4 = \frac{1}{\Delta y^4}
  \begin{bmatrix}
    2I + \frac{\sigma_1}{\Delta x^2} D^2 & -4I + 4\frac{\sigma_2}{\Delta x^2} D^2 & 2I - \frac{\sigma_2}{\Delta x^2}D^2 &    & & \\
    -2I -  \frac{\sigma\Delta y^2}{\Delta x^2} D^2 & 5I & 4 & 1  & & \\
    I & -4I & 6I & -4I & I & \\
    & \ddots & \ddots & \ddots & \ddots & \ddots \\
    & & I & -4I & 6I & -4I & I  \\
    & & & I & -4I & 5I & -2I - \frac{\sigma\Delta y^2}{\Delta x^2} D^2 \\
    & & & & 2I - \frac{\sigma_2}{\Delta x^2} D^2 & -4I + 4\frac{\sigma_2}{\Delta x^2} D^2 & 2I + \sigma_1 D^2
  \end{bmatrix}
\end{equation}
\normalsize
where $\sigma_1 := \Delta y^2 (2\sigma - 3(2 - \sigma))$ and $\sigma_2 := \Delta y^2(2 - \sigma)$. The bilaplacian matrix then will be a block pentadiagonal matrix of size $J(K+2) \times J(K+2)$, where it is defined by the sum \eqref{bilapD} using the modified matrices given by \eqref{d4x}, \eqref{d2y} and \eqref{d4y}.

\subsection{Integration over time.}

Given the definition of the function $\varphi(u)$ on Problem \eqref{main problem}, the first order derivative will be approximated using a centered finite different scheme, and the integral will be computed using a Simpson rule for each value on the $y$ coordinate. Meanwhile, the time derivative will be approximated using a finite difference scheme, analog the one used in \eqref{der2}. Finally, we will consider a Crank-Nicholson discretization for the bilaplacian; this is, we will approximate the bilaplacian over time using $\Delta ^2 \Big(\frac{U^{n+1} + U^n}{2}\Big)$.

This lead us to the numerical scheme which we will use on this work: for $U_{j,k}^n$ the numerical solution of Problem \eqref{main problem} on $(x_j,y_k,t_n)$ with $h(x,y,t) = 0$, the solution at the timestep $n+1$ will be given by
\begin{align}\label{esquema numerico}
  \bigg[(I &+ \frac{\Delta t^2}{2}D_x^4 + D_y^4 + 2D_x^2D_y^2)  U^{n+1}\bigg]_{j,k} \nonumber \\
  &= 2U_{j,k}^{n} - U_{j,k}^{n-1} - \frac{\Delta t^2}{2}[( D_x^4 + D_y^4 + 2D_x^2D_y^2)U^{n}]_{j,k}  \\
  &\; - \Delta t^2 \Big(\varphi(U_{j,k}^{n}) + a(x_j,y_k)g\Big(\frac{U_{j,k}^{n} - U_{j,k}^{n-1}}{\Delta t}\Big) \Big) \nonumber
\end{align}
if $a(x,y) = 0, \forall (x,y) \in \Omega$, and $P = S = 0$, then this scheme can control numerical diffusion of the energy if a sufficiently small value of $\Delta t$ is used. If the feedback $g(s)$ is linear, then a Newmark scheme can be used to compute the numerical solution, which will conservate the energy for any value given for $\Delta t < 1$.

This scheme was implemented on a \texttt{MATLAB} script, where the linear equation system present in \eqref{esquema numerico} was solved using the default solver of the software. When solving the static problem $\Delta ^2 u(x,y) = f(x,y)$, and using values of $\Delta x \approx 0.02$ and $\Delta y \approx 0.015$, the code can approximate the solution of the problem with errors of magnitude $10^{-6}$ for the numerical $L^2$ norm.

\subsection{Some results}

For the following experiments, we will solve Problem \eqref{main problem} using $h = 0$, $\sigma = 0.2$, $S = 10^{-5}$, $P = 10^{-3}$, $l = \frac{\pi}{150}$, and $u_1 = 0$. Function $u_0$ will be given by the solution of the following static problem

\begin{equation}\small \label{pEstatico}
\hspace*{-.2cm}
\begin{cases}
\Delta^2 u(x,y) =50 \sin(2x),\quad  \mbox{ in } \;\Omega \times (0,+\infty), \\\\

u(0,y)=u_{xx}(0,y)=u(\pi,y)=u_{xx}(\pi,y)=0, \quad \quad  \!\!\!\!(y,t)\in(-l,l), \\\\

u_{yy}(x,\pm l)+\sigma u_{xx}(x,\pm l)=0, \quad \qquad \qquad \qquad \qquad \quad x\in (0,\pi) \\\\

u_{yyy}(x,\pm l)+(2-\sigma)u_{xxy}(x,\pm l)=0,
\quad \qquad \qquad \qquad \!\! x\in(0,\pi)
\end{cases}
\end{equation}

The solution is given in \cite{Gazzola1}, Theorem 3.2. It can also be computed using this same numerical scheme. The function $a(x,y)$ is defined as follows:

\small
\begin{equation*}
  a(x,y) =
  \begin{cases}
    1, \: (x,y) \in (0,5\Delta x) \cup (\pi-5\Delta x,\pi) \times (-l, -l+5\Delta y) \cup (l-5\Delta y,l) \\
    0, \text{ otherwise. }
  \end{cases}
\end{equation*}
\normalsize
where, on the numerical scheme, $\Delta x = \frac{\pi}{150} \approx 0.02$, $\Delta y = \frac{l}{50} \approx 0.015$, and $\Delta t = 0.01$.  We will use three differents forms for the feedback function $g(s)$. Figure \ref{fig6} shows the time evolution of the energy given by equation \eqref{energy} when using $g(s) = \sqrt{s}$, while Figure \ref{fig7} shows the case when $g(s) = s$. We can see that the energy decays following the upper bounds claimed in Theorem \ref{thm:decay}. This can also be seen on Figure \ref{fig8}, when the feedback is given by
\begin{align}\label{eq5.1}
  g(s) = 
  \begin{cases}
    s^2, \quad \text{ if } s \geq 0 \\
    s^3, \quad \text{ if } s < 0
  \end{cases}
\end{align}

	\begin{figure}[h!]
		\centering
                \includegraphics[width=\linewidth]{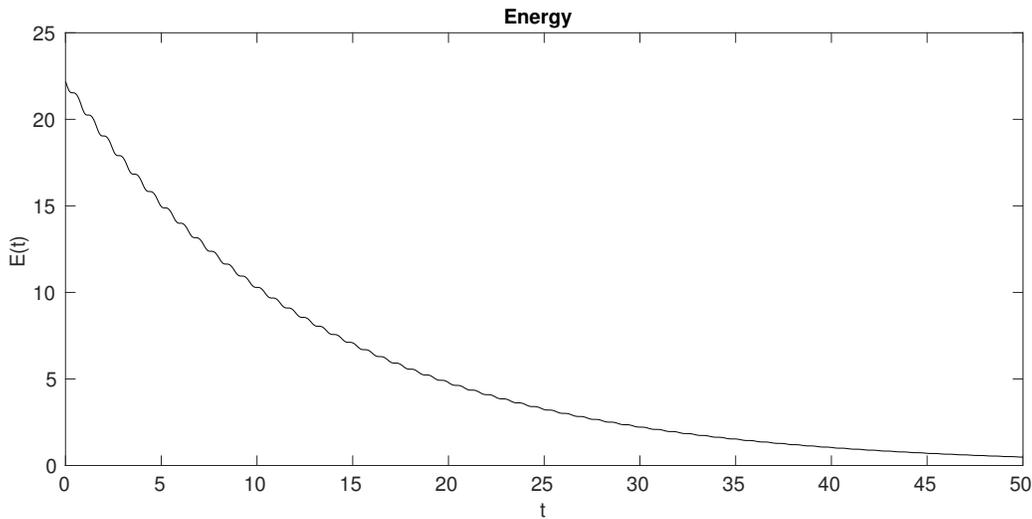}
                                \caption{Energy evolution when $g(s) = \sqrt{s}$.}
                                \label{fig6}
        \end{figure}
        	\begin{figure}[h!]
		\centering
                \includegraphics[width=\linewidth]{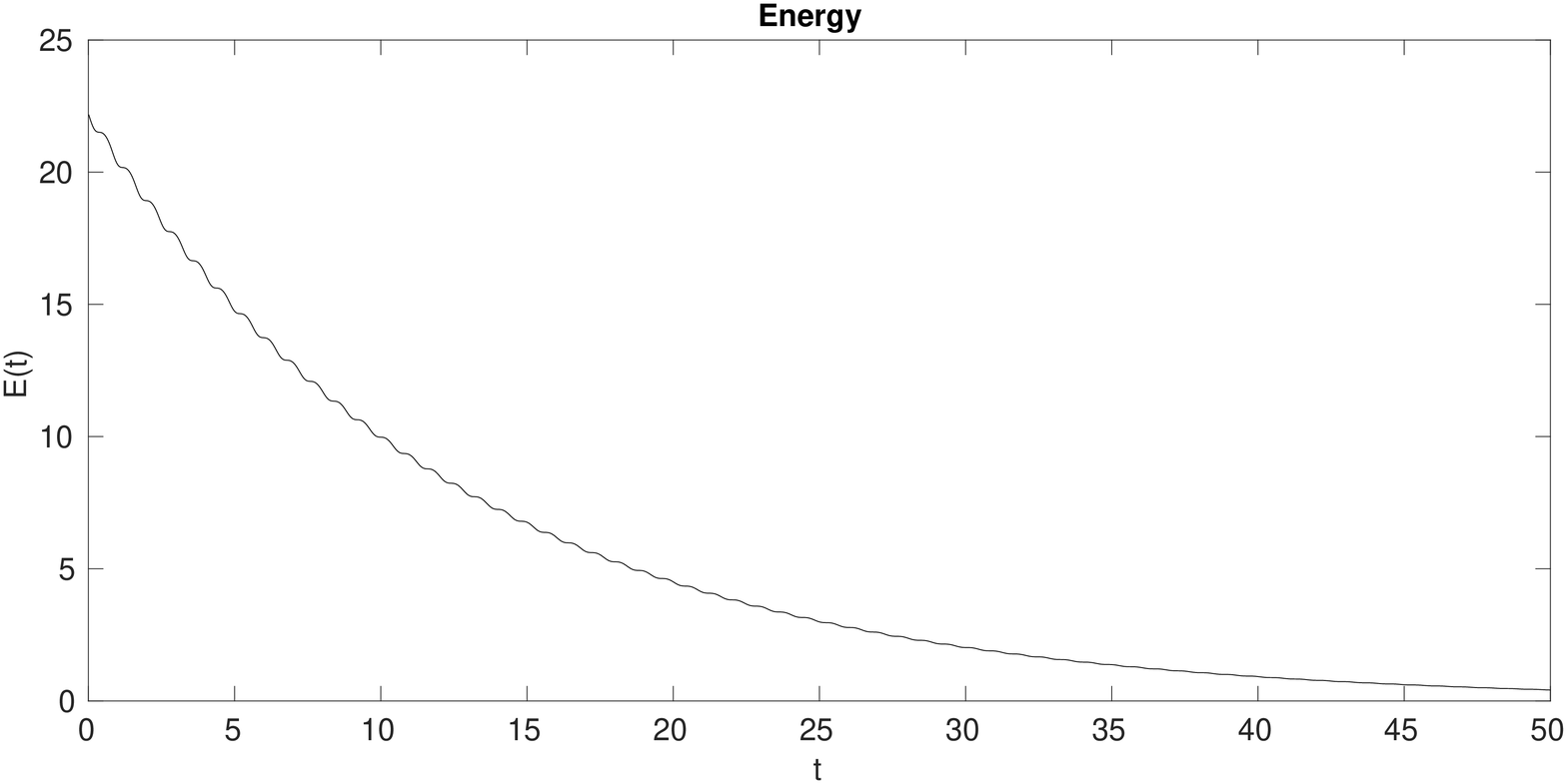}

                \caption{Energy evolution when $g(s) = s$.}
                                \label{fig7}
                \end{figure}
	\begin{figure}[h!]
		\centering
                \includegraphics[width=\linewidth]{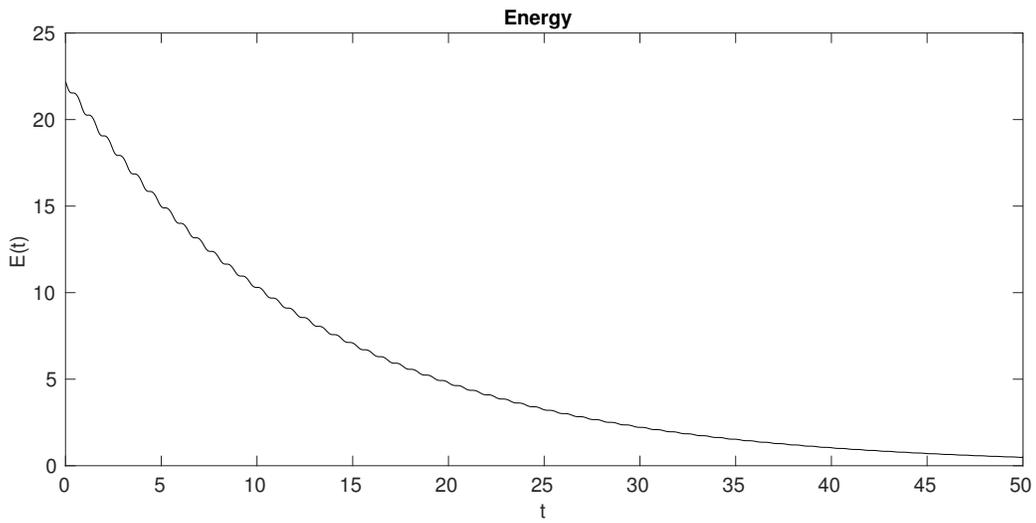}

                \caption{Energy evolution when $g(s)$ is given by Equation \eqref{eq5.1}.}
                                \label{fig8}
        \end{figure}

{
	
	\medskip 
	
	\newpage 
\section{Conclusion}

	\subsection{Analytical Part }
	
The next table presents a comparison between the present article and the existing literature regarding the problem \eqref{main problem} and similar highlighting the contributions of this paper.

\begin{table}[htb!]

	\begin{tabular}{|M{5cm}|M{2cm}|M{5cm}|}
\hline 		\multicolumn{3}{|c|}{Summary of the literature with respect to problem \eqref{main problem} and similar} \\ \hline 
Authors		&  Damping     &  Contributions   \\ \hline 
\cite{Tucsnak}		&  localized  & \hspace*{-.5cm}  \begin{minipage} [t] {0.4\textwidth} 
	\vspace{-\topsep}\begin{itemize}
		\setlength{\parskip}{0pt}
		\setlength{\itemsep}{0pt plus 1pt}
		
\crossed non  smooth domain
	\done similar model			
		\done  well-posedness
		\done stabilization
\crossed nonlinear damping
	\end{itemize} \vspace{-\topsep}
\end{minipage}
\\ \hline 
\cite{Bochicchio}		&  full  & \hspace*{-.5cm}  \begin{minipage} [t] {0.4\textwidth} 
			\vspace{-\topsep}\begin{itemize}
 \setlength{\parskip}{0pt}
 \setlength{\itemsep}{0pt plus 1pt}
	\done similar model
	\done well-posedness
\crossed  stabilization
\crossed nonlinear damping
			\end{itemize} \vspace{-\topsep}
		\end{minipage}
		    \\ \hline 
		    
		    \cite{Gazzola2}		& full   & \hspace*{-.5cm}  \begin{minipage} [t] {0.4\textwidth} 
		    	\vspace{-\topsep}\begin{itemize}
		    		\setlength{\parskip}{0pt}
		    		\setlength{\itemsep}{0pt plus 1pt}
		    	\done problem \eqref{main problem}
		    		\done well-posedness
		    		\done stabilization
		    		\crossed nonlinear damping
		    	\end{itemize} \vspace{-\topsep}
		    \end{minipage}\\\hline
\cite{Messaoudi}		& full   & \hspace*{-.5cm}  \begin{minipage} [t] {0.4\textwidth} 
	\vspace{-\topsep}\begin{itemize}
		\setlength{\parskip}{0pt}
		\setlength{\itemsep}{0pt plus 1pt}
\done similar model
			\done well-posedness
		\done stabilization
				    		\crossed nonlinear damping
		\end{itemize} \vspace{-\topsep}
		\end{minipage}\\ \hline
Present article		&  localized  & \hspace*{-.5cm}  \begin{minipage} [t] {0.4\textwidth} 
			\vspace{-\topsep}\begin{itemize}
				\setlength{\parskip}{0pt}
				\setlength{\itemsep}{0pt plus 1pt}
				
				\done non  smooth domain
				\done similar model			
				\done  well-posedness
				\done stabilization
				\done nonlinear damping
				\done to extend the unique continuation principle proved in  \cite{Kim} for domains with smooth boundary to the present case where the boundary contains corners.
			\end{itemize} \vspace{-\topsep}
		\end{minipage}
		\\ \hline 
	\end{tabular}
\caption{Existing literature regarding the problem \eqref{main problem} and similar.}
\end{table}

\newpage

\medskip

\subsection{Numerical part}We have proved new energy decay rates for some feedback functions, and those results were replicated by numerical experiments using a finite difference scheme. Given the boundary conditions of the problem, this finite difference scheme is a reasonable choice where other available finite element integrators fail. We hope this work might be of use for further studies and applications on bridges and vibrating plates.

}

\section{Appendix}
\subsection{Hessian and Laplacian}

Let $f$ be a $C^k$ function ($k \geq 2$) on a Riemannian manifold $(M,g)$. Then its Hessian with respect to the Riemanian connection $\nabla$ is given by
\[
(\nabla^2 f)(X,Y) = XY(f) - (\nabla_X Y)(f),
\]
where $X,Y$ are vector fields on $M$ and $X(f)$ is the directional derivative of $f$ with respect to the vector field $X$.
$XY(f)$ is the directional derivative of $Y(f)$ with respect to $X$.
In a coordinate system $(x_1,\ldots, x_n)$, we have that
\[
\nabla^2 f\left( \frac{\partial}{ \partial x_i},\frac{\partial}{ \partial x_j} \right) = \frac{\partial^2 f}{\partial x_i \partial x_j} - \sum_{k=1}^n \Gamma_{ij}^k \frac{\partial f}{\partial x_k},
\]
where $\Gamma_{ij}^k$ are the Christoffel symbols of $M$ with respect to $(x_1,\ldots,x_n)$.
The norm $\Vert \nabla^2 f\Vert$ is defined as
\[
\Vert (\nabla^2 f)(p)\Vert^2 =  \sum_{i,j=1}^n \left( (\nabla^2 f)(p)(e_i(p),e_j(p)) \right)^2,
\]
where $(e_1(p),\ldots,e_n(p))$ is an orthonormal basis of the tangent space $T_pM$ of $M$ at $p$.
The Laplacian of $f$ is given by
\[
\Delta f(p) = \sum_{i=1}^n (\nabla^2 f(p))(e_i(p), e_i(p)).
\]
It is straightforward that $\Vert (\nabla^2 f)(p) \Vert^2$ and $\Delta f(p)$ do not depend on the choice of the orthonomal basis $(e_1(p),\ldots,e_n(p))$.

Let $(e_1,\ldots,e_n)$ be an orthonormal moving frame and $(x_1,\ldots,x_n)$ be a coordinate system in a neighborhood $\tilde W$ of $p \in M$ such that
\begin{equation}
\label{compara referencial campos coordenados}
(e_1(p),\ldots,e_n(p))= (\partial /\partial x_1(p), \ldots, \partial /\partial x_n(p)).
\end{equation}
Due to the continuity of $\nabla^2 f$, for every $\tilde\varepsilon >0$ there exist a neighborhood $W$ of $p \in M$ such that
\begin{eqnarray}
\left( (\nabla^2 f)(q)(e_i,e_j) \right)^2
& \leq &
(1+\tilde \varepsilon)\left( (\nabla^2 f)(q)\left( \frac{\partial}{\partial x_i}, \frac{\partial}{\partial x_j} \right) \right)^2 \nonumber \\
& \leq & (1+ \tilde \varepsilon)\left( \frac{\partial^2 f}{\partial x_i \partial x_j} - \sum_{k=1}^n \Gamma_{ij}^k \frac{\partial f}{\partial x_k} \right)^2. \nonumber
\end{eqnarray}
For the sake of simplicity, we will suppose that
\begin{equation}
\label{limitante norma Hessiano}
\left( (\nabla^2 f)(q)(e_i,e_j) \right)^2  \leq  2 \left( \frac{\partial^2 f}{\partial x_i \partial x_j} - \sum_{k=1}^n \Gamma_{ij}^k \frac{\partial f}{\partial x_k} \right)^2
\end{equation}
for every $i,j=1,\ldots, n$ whenever this kind of neighborhood is needed.

Denote the distance function on $M$ by $\mathrm{dist}$.
Let $N$ be a compact and oriented submanifold of $M$ of codimension one.
The orientation of $N$ is given by a normal unit vector field $\xi$ on $N$.
A tubular neighborhood of $N$ is a subset
\[
\tilde N:=\{\tilde q\in M;\mathrm{dist}(q,N)< \varepsilon\},
\]
where $\varepsilon >0$, every $\tilde q \in \tilde N$ admits a unique $q \in N$ such that $\mathrm{dist}(q,N)=\text{dist}(q,\tilde q)$ and $\tilde q \mapsto q$ is a submersion from $\tilde N$ to $N$.
The tubular neighborhood can be constructed considering $\exp_q t\xi(q)$ for $q \in N$ and $t \in (-\varepsilon, \varepsilon)$, where $\exp_q : T_qM \rightarrow M$ is the exponential map.
We have that $\text{dist}(\exp_q t \xi(q),N)=\text{dist}(\exp_q t \xi(q),q)=\vert t\vert$ and $t$ is the oriented distance from $\exp_q t\xi(q)$ to $N$.
If $W$ is a coordinate neighborhood of $p \in N$ with coordinate system $(x_1,\ldots, x_{n-1})$, then
\[
\tilde W:=\{\exp_q t\xi(q);q \in W,t\in (-\varepsilon,\varepsilon)\}
\]
is a neighborhood of $q \in M$ with coordinate system
\begin{equation}
\label{coordenada tilde U}
(x_1,\ldots, x_{n-1},t).
\end{equation}
In what follows, we need a neighborhood
\begin{equation}
\label{tilde U}
\tilde W:=\{\exp_q t\xi(q);q \in W,t\in (-\varepsilon^\prime, \varepsilon^\prime )\}
\end{equation}
such that (\ref{limitante norma Hessiano}) is satisfied.

In this setting, we have the following result:

\begin{lemma}
	\label{estimativa local}
	Let $M$ be a Riemannian manifold and $N$ be an oriented compact submanifold of $M$ with codimension one.
	Let $p \in M$ and consider a neighborhood $\tilde W$ of $p \in M$ with coordinate system $(x_1,\ldots,x_{n-1},x_n=t)$ as in (\ref{coordenada tilde U}) and (\ref{tilde U}) and satisfying (\ref{compara referencial campos coordenados}) and (\ref{limitante norma Hessiano}) with respect to an orthonormal frame $(e_1,\ldots,e_n)$.
	Set the smooth function $\eta:\tilde W \rightarrow \mathbb R$ defined in this coordinate system as $\eta (x_1,\ldots, x_{n-1},x_n) = x_n^4$.
	Then
	\[
	\frac{\Vert \nabla^2 \eta \Vert^2}{\vert \eta \vert}\leq 30
	\]
	on $\tilde W \backslash N$, after an eventual further shrinking of $\tilde W$.
\end{lemma}

\noindent{\bf Proof:}
	
	Due to the properties of $(\tilde W,(x_1,\ldots,x_n))$, we have that
	\begin{eqnarray}
	\nabla^2 \eta (e_i,e_j) &  = & 2 \left( \frac{\partial^2 \eta}{\partial x_i \partial x_j}-\sum_{ij}^k \Gamma_{ij}^k \frac{\partial \eta}{\partial x_k} \right)
	=
	24 \delta_{in} \delta_{jn} x_n^2 - 8 \Gamma_{ij}^n x_n^3 \nonumber \\
	& = & x_n^2 \left(24 \delta_{in} \delta_{jn}- 8 \Gamma_{ij}^n x_n \right),
	\end{eqnarray}
	where $\Gamma^k_{ij}$ are the Christoffel symbols of $(\tilde W,(x_1,\ldots,x_n))$.
	Then
	\[
	\Vert (\nabla^2 \eta)(p)\Vert^2 =  \sum_{i,j=1}^n \left( (\nabla^2 \eta)(p)(e_i,e_j) \right)^2 \leq 30 x_n^4,
	\]
	for an eventually smaller $\tilde W$ (we consider $\tilde W$ such that $\Gamma^k_{ij}$ is bounded and $\varepsilon$ is sufficiently small).
	Thus
	\[
	\frac{\Vert \nabla^2 \eta \Vert^2}{\vert \eta \vert}\leq 30
	\]
	on $\tilde W$.\,$ \qedsymbol $

\begin{thm}
	\label{eta em N til}
	Let $M$ be a Riemannian manifold and and $N$ be an oriented compact submanifold of $M$ with codimension one.
	Then there exist a tubular neighborhood $\tilde N$ of $N$ and a smooth function $\eta:\tilde N \rightarrow \mathbb R$ such that
	\[
	\frac{\Vert \nabla^2 \eta \Vert^2}{\vert \eta \vert}\leq 30
	\]
	on $\tilde N \backslash N$.
\end{thm}

\noindent{\bf Proof:}
	
	Cover $N$ by a finite family of open subsets $\tilde W$ as in Lemma \ref{estimativa local}.
	Let $\varepsilon >0$ be the minimum of all $\varepsilon$ correspondent to each $\tilde W$ and let $\tilde N$ be the $\varepsilon$-tubular neighborhood of $N$.
	Then $\eta: \tilde N \rightarrow \mathbb R$, defined locally as in Lemma \ref{estimativa local}, is well defined because $\eta$ is the oriented distance from $x$ to $N$.
	
	\smallskip
	
	Therefore $\eta$ satisfies
	\[
	\frac{\Vert \nabla^2 \eta \Vert^2}{\vert \eta \vert}\leq 30
	\]
	on $\tilde N\backslash N$.
\,$ \qedsymbol $

\begin{lemma}
	\label{existencia suave intermediario}
	Let $M$ be a differentiable manifold and let $\hat U, \hat V \subset M$ be closed disjoint subsets.
	Then there exist open subsets $U$ and $V$ with smooth boundaries containing $\hat U$ and $\hat V$ respectively, with $\bar U \cap \bar V = \emptyset$.
\end{lemma}
\noindent{\bf Proof:}
	
	Due to the smooth Urysohn lemma, there exist a smooth function $\varphi:M \rightarrow \mathbb R$ such that $\varphi\vert_{\hat U}\equiv 1$ and $\varphi\vert_{\hat V} \equiv 0$ (see \cite{Colon}).
	Let $a,b\in (0,1)$, with $a<b$, be regular values of $\varphi$.
	Then $U:=\varphi^{-1}([0,a))$ and $V:=\varphi^{-1}((b,1])$ satisfy the conditions stated in the lemma.
$\qedsymbol $

\begin{thm}
	\label{vizinhanca tubular em fechados}
	Let $M$ be a compact and connected Riemannian manifold, eventually with boundary, and let $U$ and $V$ be open subsets of $M$ with smooth boundaries such that $\bar U \cap \bar V = \emptyset$.
	Suppose that $\partial M \subset U$.
	Then there exist a smooth function $\eta : M \rightarrow [0,1]$ such that $\eta\vert_{\bar U} \equiv 1$, $\eta\vert_{\bar V} \equiv 0$, $\eta(x)\in (0,1)$ if $x \in M\backslash (\bar U \cup \bar V)$ and
	\[
	\frac{\Vert \nabla^2 \eta\Vert^2}{\vert \eta \vert}
	\]
	is bounded in $M \backslash \bar V$.
\end{thm}

\noindent{\bf Proof:}
	Denote $\tilde \partial U = \partial U \backslash \partial M$.
	Then $\tilde \partial U$ and $\partial V$ are disjoint compact submanifolds of the boundaryless Riemannian manifold $\text{int}M$.
	Observe that they are orientable and we choose the normal vector field pointing outside $U$ and $V$ respectively.
	Let $\tilde U$ and $\tilde V$ be $\varepsilon$-tubular neighborhoods of $\tilde \partial U$ and $\partial V$ respectively (as submanifolds of $\text{int}M$) and set $A_1=U \cup \tilde U$ and $A_2 = V \cup \tilde V$.
	We can choose $\varepsilon > 0 $ such that $\bar A_1 \cap \bar A_2 = \emptyset$ and such that $\tilde U$ and $\tilde V$ satisfy the conditions of Theorem \ref{eta em N til}.
	Let $A_0=\{x \in M;\text{dist}(x,U)>\varepsilon/2,\text{dist}(x,V) > \varepsilon/2\}$.
	Then $\{A_0,A_1,A_2\}$ is an open cover of $M$ and we consider a smooth partition of unity $\{\phi_0,\phi_1,\phi_2\}$ subordinated to $\{A_0,A_1,A_2\}$.
	Let $\tilde \eta_1:\tilde U \rightarrow \mathbb R$ and $\tilde \eta_2: \tilde V \rightarrow \mathbb R$ be the oriented distance to $\tilde \partial U$ and $\partial V$ respectively.
	Define $\eta_0:A_0 \rightarrow \mathbb R$ as the constant function $\eta_0 \equiv 1/2$, $\eta_1:A_1 \rightarrow \mathbb R$ by
	\[
	\eta_1(x) =
	\left\{
	\begin{array}{ccc}
	1 & \text{if} & x \in \bar U \\
	1- \tilde\eta_1^4(x) & \text{if} & x \in \tilde U\backslash \bar U
	\end{array}
	\right.
	\]
	and $\eta_2:A_2 \rightarrow \mathbb R$ by
	\[
	\eta_2(x) =
	\left\{
	\begin{array}{ccc}
	0 & \text{if} & x \in \bar V \\
	\tilde\eta_2^4(x) & \text{if} & x \in \tilde V \backslash \bar V.
	\end{array}
	\right.
	\]
	These functions are of class $C^4$.
	Define $\eta=\phi_0 \eta_0 + \phi_1 \eta_1 + \phi_2 \eta_2$.
	$\eta$ is also of class $C^4$ and it is equal to $\eta_2^4$ in an $\varepsilon/2$ neighborhood $A_3$ of $\tilde \partial V$ because $\phi_0$ and $\phi_1$ are zero there.
	Therefore
	\[
	\frac{\Vert \nabla^2 \eta \Vert^2}{\vert \eta \vert}\leq 30
	\]
	on $A_3\backslash \bar V$ due to Theorem \ref{eta em N til} and of course
	\[
	\frac{\Vert \nabla^2 \eta \Vert^2}{\vert \eta \vert}
	\]
	is bounded in the compact subset $M\backslash (A_3\cup V)$ because $\vert \eta \vert$ never vanishes there.
	Therefore
	\[
	\frac{\Vert \nabla^2 \eta \Vert^2}{\vert \eta \vert}
	\]
	is bounded on $M\backslash \bar V$.
$\qedsymbol $


\begin{remark}
	\label{laplacian and other derivatives}
	If $(e_1,\ldots,e_n)$ is an orthonormal moving frame on $M$, then Theorem \ref{vizinhanca tubular em fechados} implies that the quotients
	\[
	\frac{(\nabla^2 \eta(e_1,e_j))^2}{\vert \eta\vert}
	\]
	are bounded in $M\backslash \bar V$ for $i,j=1,\ldots,n$.
	This fact implies that
	\[
	\frac{\vert \Delta \eta \vert^2}{\vert \eta \vert}
	\]
	is bounded in $M\backslash \bar V$ as well.
	In particular if $M$ is a subset of $\mathbb R^n$ (with the canonical metric) and $(x_1,\ldots,x_n)$ is the canonical coordinate system of $\mathbb R^n$, then all partial derivatives
	\[
	\frac{\vert \eta_{x_ix_j}\vert^2}{\vert \eta \vert}
	\]
	are bounded in $M \backslash \bar V$.
\end{remark}

\begin{remark}
	\label{observacao derivada primeira}
	\[
	\frac{\Vert \nabla \eta \Vert^2}{\vert \eta\vert}
	\]
	is bounded in $M\backslash \bar V$.
	In fact, just notice that $\Vert \nabla \eta (p) \Vert^2= \sum_{i=1}^n ((\nabla \eta (e_i))(p))^2$ for an orthonormal basis $(e_1(p),\ldots,e_n(p))$ of $T_pM$ and make the same calculations that we made with the Hessian.
	In particular, if $M\subset \mathbb R^n$ and $(x_1,\ldots,x_n)$ are the canonical coordinates, then we have that
	\[
	\frac{\vert \eta_{x_i}\vert^2}{\vert \eta\vert}
	\]
	are bounded on $M\backslash \bar V$ for every $i=1,\ldots,n$.
\end{remark}

\subsection{Smoothing vertices}

Let $\partial \Omega$ be the boundary of the rectangle.
Let $A$ one of its vertices. Then the neighborhood of $A \in \partial \Omega$ is can be identified with the graph of $\varphi(t)= \vert t\vert$.
Let $\eta:\mathbb R \rightarrow \mathbb R$ be the standard mollifier with support $[-1,+1]$ and define $\eta_\varepsilon (t) = \frac{\eta(t/\varepsilon)}{\varepsilon}$.
If we apply the mollifier smoothing on $\varphi$, then we have the following result, which is enough for our purposes.

\begin{prp}\label{smooth}
	The function
	\[
	\varphi_\varepsilon (t):=\int_{-\infty}^{\infty} \eta_\varepsilon (t-s)\varphi(s)ds
	\]
	has the following properties:
	
	\begin{enumerate}
		\item $\varphi_\varepsilon$ is smooth. Moreover $\varphi_\varepsilon$ converges to $\varphi$ uniformly;
		\item $\varphi_\varepsilon(t)= \varphi (t)$ outside $[-\varepsilon, \varepsilon]$;
		\item $\varphi_\varepsilon(t)>\varphi (t) $ for $t \in (-\varepsilon,\varepsilon)$.
	\end{enumerate}
\end{prp}
\noindent{\bf Proof:}
	
	Item (1) is a classical result;
	
	Item (2): Suppose that $t \in [\varepsilon,\infty)$ (the case $t \in (-\infty,\varepsilon]$ is analogous).
	Then
	\begin{eqnarray}
	\varphi_\varepsilon (t)
	& = & \int_{t-\varepsilon}^{t+\varepsilon} \eta_\varepsilon (t-s)\varphi(s)ds
	= \int_{t-\varepsilon}^{t+\varepsilon} \eta_\varepsilon (t-s) s ds \nonumber \\
	& = & \int_{t-\varepsilon}^{t+\varepsilon} \eta_\varepsilon (t-s) s ds
	= \int_{-\varepsilon}^{\varepsilon} \eta_\varepsilon (s) (t-s) ds = t. \nonumber
	\end{eqnarray}
	
	Item (3): Suppose that $t\in [0,\varepsilon)$ (the case $t \in (-\varepsilon,0]$ is analogous).
	Then
	\begin{eqnarray}
	\varphi_\varepsilon (t)
	& = & \int_{t-\varepsilon}^{t+\varepsilon} \eta_\varepsilon (t-s)\varphi(s)ds
	> \int_{t-\varepsilon}^{t+\varepsilon} \eta_\varepsilon (t-s) s ds = t \nonumber
	\end{eqnarray}
	where the strict inequality holds because $\mathrm{supp}\eta_\varepsilon=[-\varepsilon,\varepsilon]$ and $\varphi (s) > s $ in a set of positive measure.
$\qedsymbol $

Finally we can use the graph of $\varphi_\varepsilon$ as the boundary of domains that approximate $\Omega$ uniformly.

\medskip

\medskip




\end{document}